\documentclass[11pt,english]{article}
\usepackage[T1]{fontenc}
\usepackage{verbatim}
\usepackage{graphicx, subfigure}
\expandafter\def\csname ver@subfig.sty\endcsname{}
\usepackage[latin9]{inputenc}
\usepackage{geometry}
\usepackage{amsmath}
\usepackage{appendix}
\usepackage{amssymb}
\usepackage{pifont}
\usepackage{esint}
\usepackage{makecell,multirow,diagbox}
\usepackage{mathtools}
\usepackage{epstopdf}
\usepackage[latin9]{inputenc}
\usepackage{pict2e}
\usepackage{xcolor}%定义了一些颜色
\usepackage{colortbl,booktabs}%第二个包定义了几个*rule
\usepackage{stmaryrd}
\usepackage{esint}
\usepackage[all,cmtip]{xy}
\usepackage{mathtools}
\usepackage{float}
\usepackage{setspace}
\usepackage{authblk}
\usepackage{amsthm}
\usepackage{mathrsfs}
\usepackage{stmaryrd}
\usepackage{bbold}
\usepackage{color}
\usepackage{tikz}
\usepackage[all,cmtip]{xy}
\usepackage{algorithm}  
\usepackage{algorithmic}

\theoremstyle{plain}\newtheorem{theorem}{Theorem}[section]
\theoremstyle{plain}\newtheorem{lemma}{Lemma}[section]
\newtheorem{pro}{Proposition}
\newtheorem{remark}{Remark}
\usepackage{color}
\definecolor{marin}{rgb} {0., 0.3, 0.7}
\definecolor{rouge}{rgb} {0.8, 0., 0.}
\definecolor{sepia}{rgb} {0.8, 0.5, 0.}
\usepackage[colorlinks,citecolor=sepia,linkcolor=marin,bookmarksopen,bookmarksnumbered]{hyperref}

\theoremstyle{definition}

\DeclareSymbolFont{largesymbol}{OMX}{yhex}{m}{n}
\DeclareMathAccent{\Widehat}{\mathord}{largesymbol}{"62}

\makeatletter

\usepackage{babel}
\usepackage{subfig}
\graphicspath{{Dirac/}}

\def\be{\begin{equation}}
\def\ee{\end{equation}}

\makeatother
\DeclareMathSizes{14}{14}{9.8}{7}
\usepackage{lipsum}

\usepackage{epsfig}
\newdimen\GGGlength
\newdimen\GGGheight
\newbox\GGGbox
\def\GGGput[#1,#2](#3,#4)#5{%
  \setbox\GGGbox\vbox{\hbox{#5}\kern0pt}%
  \GGGlength\wd\GGGbox%
  \divide\GGGlength by100 \multiply\GGGlength by#1%
  \GGGheight\ht\GGGbox%
  \divide\GGGheight by100 \multiply\GGGheight by#2%
  \put(#3,#4){\kern-\GGGlength\raise-\GGGheight\box\GGGbox}}

\begin{document}

\title{Semi-Lagrangian methods for a plasma hybrid model with multi-species kinetic ions and massless electrons}
\date{}
\author[1]{Yingzhe Li}
\author[2]{Philip J. Morrison}
\author[1,3]{Eric Sonnendr\"ucker}
\affil[1]{Max Planck Institute for Plasma Physics, Boltzmannstrasse 2, 85748 Garching, Germany}
\affil[2]{Department of Physics and Institute for Fusion Studies, The University of Texas at Austin, Austin, TX78712, USA}
\affil[3]{Technical University of Munich, Department of Mathematics, Boltzmannstrasse 3, 85748 Garching, Germany}
\maketitle
%##########################
\begin{abstract}
The semi-Lagrangian methods with the improved number of one-dimensional advections are proposed for a plasma hybrid model with kinetic ions and mass-less electrons.  
Two subsystems with mass, momentum, and energy conservation are obtained by a Poisson bracket-based splitting method. For the subsystem in which the distribution functions and the fields are coupled, the second order and reversible modified implicit mid-point rule is used in time with 
 the specially designed mean velocity. The distribution functions are not involved in the iterations and are solved by exact splittings with only one dimensional advections, which makes the proposed schemes efficient. The cancellation problem is overcome by the numerical schemes constructed. Moreover, for the case with a periodic boundary condition, the magnetic field obtained is divergence free, mass, momentum, and energy are conserved. The methods can be extended to cases with multiple ion species.

\end{abstract}
\setcounter{tocdepth}{2} %to only see subsections

%\tableofcontents

\section{Introduction}

In this paper, we consider the semi-Lagrangian methods~\cite{semi1,semi2} with the improved number of one dimensional advections for a plasma hybrid model~\cite{YinL,cai2021hybrid} with kinetic ions and mass-less electrons: 
\begin{equation}\label{model}
\begin{aligned}
\text{kinetic ions:}\qquad &\frac{\partial f}{\partial t} + {\boldsymbol v} \cdot \frac{\partial f}{\partial {\boldsymbol x}} + \frac{q}{m} ({\boldsymbol E} + {\boldsymbol v} \times {\boldsymbol B}) \cdot \frac{\partial f}{\partial {\boldsymbol v}} = 0\,,
\\[2mm]
\text{Faraday's law:}\qquad &\frac{\partial {\boldsymbol B}}{\partial t} = - \nabla \times {\boldsymbol E}\,,\qquad \nabla \cdot {\boldsymbol B} = 0\,,
\\[1mm]
\text{Ohm's law:}  \qquad & {\boldsymbol E} = - \frac{\nabla p}{\rho} - \left(\boldsymbol u - \frac{\boldsymbol J}{\rho}  \right) \times {\boldsymbol B} \,.
\end{aligned}
\end{equation}
Here,  
  $f(t, {\boldsymbol x}, {\boldsymbol v})$ denotes the ion distribution function depending on time $t \in \mathbb{R}$, position ${\boldsymbol x} \in \mathbb{R}^3$ and velocity ${\boldsymbol v} \in \mathbb{R}^3$, each ion has charge $q$ and mass $m$, $\rho$ is the electron charge density with $\rho = \mathrm{e} n_e$, where $n_\mathrm{e}$ is the quasi-neutrality number density, which is equal to $n = \frac{q}{\mathrm{e}} \int f\,\mathrm{d}{\boldsymbol v} $ with the electron charge $-\mathrm{e}$, ${\boldsymbol u} = q \int {\boldsymbol v} f \mathrm{d}{\boldsymbol v}/\rho$ is the ion current carrying drift velocity,
${\boldsymbol E}(t, {\boldsymbol x})$ and ${\boldsymbol B}(t, {\boldsymbol x})$ stand for the electromagnetic fields, ${\boldsymbol J} = \frac{1}{\mu_0} \nabla \times {\boldsymbol B}$ denotes the plasma current with the vacuum permeability $\mu_0$, and $p$ is the pressure. When the electrons are isothermal or adiabatic, the pressure is determined by the density as $p = k_\mathrm{B}T_\mathrm{e} n_\mathrm{e}^\gamma$, where $k_\mathrm{B}$ is the Boltzmann constant, $T_\mathrm{e}$ is the electron temperature, and $\gamma$ is the ratio of specific heat. Specifically, $\gamma =1$ for isothermal electrons and $\gamma \neq 1$ for adiabatic electrons~\cite{chacon1}.  As mentioned in~\cite{le}, for more complex systems with gradients in the initial conditions,  it is necessary to include the following separate evolution equation for the electron pressure $p$~\cite{chacon1},
\begin{equation}\label{eq:firstpressure}
\frac{\partial p}{\partial t} + \nabla \cdot ({\boldsymbol u}_\mathrm{e}\, p) + (\gamma - 1)\, p\, \nabla \cdot {\boldsymbol u}_\mathrm{e} = 0, \quad \gamma \neq 1,
\end{equation}
where ${\boldsymbol u}_\mathrm{e} = {\boldsymbol u} - \frac{\boldsymbol J}{\rho}$ is the electron current carrying drift velocity. For isothermal/adiabatic electrons, the electron response is assumed to be fast and local, such that the pressure always follows the density according to an adiabatic law, 
electron pressure equation is more general and can be extended when non-adiabatic effects (e.g. heat transport, kinetic corrections) are important~\cite{chacon1,le,YinL}.

Plenty of particle-in-cell methods~\cite{current2D,Pegasus, CAMELIA,chacon1,yingzhe1,yingzhe2} and grid based methods~\cite{valentini,vlasiator,palmroth2018vlasov} have been proposed for this hybrid model.  Here, we focus on the latter, and use the semi-Lagrangian method~\cite{semi1,semi2} for solving the Vlasov equation.
One challenge for solving this hybrid model is to choose a suitable mean fluid velocity of the ions to advance the velocity distribution of the ions, such that the schemes have good conservation properties and are also efficient. In~\cite{valentini,vlasiator,palmroth2018vlasov}, the current advance method~\cite{current2D} and exact splitting methods~\cite{chenbao,welling} with moving (back) the frame are used to treat the nonlinearity of the mean velocity, and give efficient methods by solving the rotation part of the Vlasov equation as one dimensional advections. However, the methods in~\cite{valentini,vlasiator,palmroth2018vlasov} do not preserve mass, momentum, energy and other properties at the same time. 
The main improvement in current work includes, 1. The pressure term is considered in the exact splitting~\cite{exact1,exact2} and the specially designed mean velocity is used to overcome the cancellation problem~\cite{cancellation}; 2. The mid-point rule is used for the magnetic field and pressure, such that the mass, momentum, and energy are conserved.

%Motivated by the two challenges of the hybrid model~\cref{model} and the recent advancements in particle-in-cell methods~\cite{chacon1, yingzhe1, yingzhe2, GEMPIC,jianyuan, heyang,martin,struphy,dual}, this work aims to develop efficient grid-based numerical schemes for the hybrid model~\cref{model} that exhibit strong conservation properties.

The semi-Lagrangian methods with conservation properties for kinetic equations have been proposed and developed in the literature. We list some references here. The mass and momentum conserving properties of the cubic spline-based semi-Lagrangian method~\cite{semi1} for the Vlasov--Poisson equations have been given in~\cite{bookvlasov}. 
Conservative semi-Lagrangian methods have been developed, such as in~\cite{filbet,crouseilles2010conservative,xiong2014high}, in which the positivity of the distribution functions can be guaranteed when appropriate flux limiters are used. And
efficient semi-Lagrangian methods that conserve mass and energy have been constructed for the Vlasov--Amp\`ere equations in~\cite{hongtao} using the conservative semi-Lagrangian methods~\cite{filbet, xiong2014high} without filters. In this work, we choose to use the cubic spline based semi-Lagrangian methods~\cite{semi1} which has good conservation properties as mentioned in~\cite{bookvlasov}. 

In order to reduce the cost associated with the high-dimensional distribution functions,
for semi-Lagrangian methods~\cite{semi1, semi2}, splitting techniques~\cite{splitting}, such as directional splitting~\cite{directional,valentini} and Hamiltonian splitting~\cite{yingzhe6d, Nicolas, he, comment}, are employed to avoid high-dimensional interpolations, thereby reducing the problem to one-dimensional advections. A powerful technique is the exact splittings~\cite{exact1,exact2}, which does not have time discretization error, and  reduces the rotation part of the Vlasov equation, i.e., 
$\frac{\partial f}{\partial t}  + ({\boldsymbol v} \times {\boldsymbol B}) \cdot \frac{\partial f}{\partial {\boldsymbol v}} = 0$
into one dimensional advections.
In order to reduce the number of one dimensional advections, the splitting methods~\cite{splitting,Strang,HLW} with smaller number of sub-systems are preferable, motivated by which, in this work we present an exact splitting for $\frac{\partial f}{\partial t}  + ({\boldsymbol E} + {\boldsymbol v} \times {\boldsymbol B}) \cdot \frac{\partial f}{\partial {\boldsymbol v}} = 0$ by decomposing the electric field ${\boldsymbol E}$ as the sum of the perpendicular and parallel parts of the magnetic field ${\boldsymbol B}$.

Thanks to the exact splitting constructed for $\frac{\partial f}{\partial t}  + ({\boldsymbol E} + {\boldsymbol v} \times {\boldsymbol B}) \cdot \frac{\partial f}{\partial {\boldsymbol v}} = 0$, the Poisson splitting~\cite{2020en} based on the Poisson bracket~\cite{Poisson, Tronci}  
gives only two sub-steps when the electron pressure equation~\eqref{eq:firstpressure} is included in the model~\eqref{model}. 
For the sub-step in which the distribution functions are coupled with the magnetic field and pressure, the second order and reversible modified mid-point rule with the exact splittings~\cite{exact2, exact3, exact1} are used with a specifically chosen mean velocity, i.e., the time average of itself, such that mass, momentum, energy are conserved, and the cancellation problem~\cite{cancellation} is overcome.
And the Fourier spectral methods are used for the discretization of the fields.
 The proposed schemes are efficient in the sense that only one dimensional advections are needed for the update of the distribution functions thanks to the exact splitting, and the distribution functions are not involved in the iterations, which are only about the magnetic fields, pressure, and the moments of the distribution functions due to the good property of the cubic spline based semi-Lagrangian methods~\cite{semi1} for the velocity moments of the distribution functions. The other sub-step only involves the spatial advections, which also conserves mass, momentum, and energy exactly when Fourier spectral method~\cite{jieshen} or cubic spline-based semi-Lagrangian methods~\cite{semi1} are used.
 
With the time discretization and the specially chosen mean velocity mentioned above, all the above properties also hold when higher order spline based semi-Lagrangian methods~\cite{semi1,bookvlasov} are adopted for the distribution functions, the central finite difference method is used for the fields~\cite{chacon1}, and there are multi-species ions. 

The paper is organized as follows. In Section~\ref{sec:Poissonbrackets}, we introduce the Poisson bracket and describe the sub-steps defined by the Poisson splitting method. In Section~\ref{sec:Discretization}, the discretizations by Fourier spectral method and semi-Lagrangian method are conducted. In Section~\ref{sec:iso}, we show the schemes of the two sub-steps and the conservation properties.
Section~\ref{sec:reduced} presents a reduced model with one spatial dimension and two velocity dimensions, which is employed for numerical simulations. In section~\ref{sec:Numerical}, we validate the time accuracy order and the reversibility of the time semi-discretization, simulate the Landau damping and Bernstein waves, and present the errors of mass, momentum, and energy. 
In section~\ref{sec:conclusion}, we conclude the paper with a summary and an outlook to future works.

\section{Poisson brackets}\label{sec:Poissonbrackets}
As~\cite{yingzhe1} we normalize the hybrid model~\eqref{model} with the following characteristic scales,
\begin{equation*}
\begin{aligned}
&t = \frac{t'}{\Omega}\,, \quad {\boldsymbol v} = {\boldsymbol v'} v_\mathrm{A}\,,
\quad {\boldsymbol x} = {\boldsymbol x}' \frac{v_\mathrm{A}}{\Omega}\,,\quad {\boldsymbol B} = {\boldsymbol B}' B_0\,,\quad {\boldsymbol E} = {\boldsymbol E}' v_\mathrm{A} B_0\,, \quad p = p' mn_0v_\mathrm{A}^2,\\
& \rho = \rho' \mathrm{e}n_0\,,\quad \boldsymbol u = {\boldsymbol u}' v_\mathrm{A}\,,\quad \kappa := \frac{k_\mathrm{B} T_\mathrm{e}}{m v_\mathrm{A}^2}\,,\quad \boldsymbol J = \boldsymbol J' \mathrm{e} n_0 v_\mathrm{A}\,,\quad \mathcal H = \mathcal H' m n_0v_\mathrm{A}^2,
\end{aligned}
\end{equation*}
where primed quantities are dimension-less, $\Omega = qB_0/m$ is the ion cyclotron frequency, $B_0$ is the characteristic magnetic field strength, $v_\mathrm{A} = B_0/\sqrt{\mu_0 m n_0}$ denotes the Alfv\'en velocity with $n_0$ the characteristic particle density, and the parameter $\kappa$ denotes the ratio of $k_\mathrm{B} T_\mathrm{e}$ to $m v_\mathrm{A}^2$. Then we obtain the following dimensionless hybrid model with general electrons in the case of single ion species with unit charge~\footnote{We omit the primes on normalized quantities for simplicity in the remaining parts of this work.},
\begin{equation}
\label{eq:dimensionlessmodel}
\begin{aligned}
& \frac{\partial f}{\partial t} + {\boldsymbol v} \cdot \frac{\partial f}{\partial {\boldsymbol x}} +({\boldsymbol E} + {\boldsymbol v} \times {\boldsymbol B}) \cdot \frac{\partial f}{\partial {\boldsymbol v}} = 0\,,\\
&\frac{\partial {\boldsymbol B}}{\partial t} = - \nabla \times {\boldsymbol E}\,,\quad \nabla \cdot {\boldsymbol B} = 0,\, \quad {\boldsymbol u} = {\int {\boldsymbol v}f\, \mathrm{d}{\boldsymbol v}} \Big/\rho, \quad \rho = {\int f\, \mathrm{d}{\boldsymbol v} }, \\
& {\boldsymbol E} = - \frac{\nabla p}{\rho} - \left(\boldsymbol u - \frac{\boldsymbol J}{\rho}  \right) \times {\boldsymbol B}, \quad \boldsymbol J = \nabla \times {\boldsymbol B},\\
& \frac{\partial p}{\partial t} + \nabla \cdot ({\boldsymbol u}_\mathrm{e}\, p) + (\gamma - 1)\, p\, \nabla \cdot {\boldsymbol u}_\mathrm{e} = 0, \quad {\boldsymbol u}_\mathrm{e} = {\boldsymbol u} - \frac{\boldsymbol J}{\rho}, \quad \gamma \neq 1.
\end{aligned}
\end{equation}
where $\rho=\int f\, \mathrm{d}{\boldsymbol v}$ due to the quasi-neutrality condition. In the case with the isothermal/adiabatic electrons, the pressure $p$ is determined with the density $\rho$ as $p = \kappa \rho^\gamma$. 

In the following, we focus on the case with general electrons, the case with isothermal or adiabatic electrons can be considered similarly. 
 There is a Poisson bracket for model~\eqref{eq:dimensionlessmodel} proposed in~\cite{yingzhe1, Poisson} based on the results in~\cite{Tronci}. 
The Poisson bracket $\{\cdot,\cdot\}: V \times V \to V$, for the model~\eqref{eq:dimensionlessmodel} is, where $V$ denotes the vector space of functionals $\mathcal{F}: (f, {\boldsymbol B}, {\boldsymbol p}) \mapsto \mathbb R$ on the space of unknowns,
\begin{equation}\label{bra0}
\{\mathcal{F},\mathcal{G}\} = \{\mathcal{F},\mathcal{G} \}_{xv} + \{\mathcal{F},  \mathcal{G}\}_{pvb},
\end{equation}
\begin{equation*}
\begin{aligned}
&\{  \mathcal{F},  \mathcal{G} \}_{pvb} =  -\int f({\boldsymbol v})\, f({\boldsymbol v}')\, \frac{\boldsymbol B}{\rho}  \cdot \left( \nabla_v \frac{\delta  \mathcal{F}}{\delta f} \times \nabla_{v'} \frac{\delta \mathcal{G}}{\delta f} \right)\,\mathrm{d}{\boldsymbol x}\,  \mathrm{d}{\boldsymbol v}\, \mathrm{d}{\boldsymbol v}' \\
& + \int f  {\boldsymbol B} \cdot \nabla_v \frac{\delta  \mathcal{F}}{\delta f} \times \nabla_v \frac{\delta \mathcal{G}}{\delta f}\,  \mathrm{d}{\boldsymbol x}\,  \mathrm{d}{\boldsymbol v}  -\int \frac{\boldsymbol B}{\rho} \cdot \left(\nabla \times \frac{\delta  \mathcal{F}}{\delta {\boldsymbol B}}\right) \times \left(\nabla \times \frac{\delta \mathcal{G}}{\delta {\boldsymbol B}}\right)  \mathrm{d}{\boldsymbol x}\\
&+ \int \frac{f}{\rho}  \boldsymbol{B} \cdot \left[ \nabla_v \frac{\delta  \mathcal{F}}{\delta f} \times \left(\nabla \times \frac{\delta \mathcal{G}}{\delta {\boldsymbol B}}\right) -   \nabla_v \frac{\delta \mathcal{G}}{\delta f} \times \left(\nabla \times \frac{\delta  \mathcal{F}}{\delta {\boldsymbol B}}\right)\right] \mathrm{d}{\boldsymbol x}\,  \mathrm{d}{\boldsymbol v} \\
& +\int \gamma p \left( \frac{\delta \mathcal{F}}{\delta p} \nabla \cdot \frac{\nabla \times \frac{\delta \mathcal{G}}{\delta {\boldsymbol B}} - \int f\, \nabla_v \frac{\delta \mathcal{G}}{\delta f}\, \mathrm{d}{\boldsymbol v}}{\rho} - \frac{\delta \mathcal{G}}{\delta p} \nabla \cdot \frac{\nabla \times \frac{\delta \mathcal{F}}{\delta {\boldsymbol B}} - \int f\, \nabla_v \frac{\delta \mathcal{F}}{\delta f}\, \mathrm{d}{\boldsymbol v}}{\rho} \right) \mathrm{d}{\boldsymbol x} \\
 &+ \int \frac{\nabla p}{\rho} \cdot \left( \frac{\delta \mathcal{F}}{\delta p} \left( \nabla \times \frac{\delta \mathcal{G}}{\delta {\boldsymbol B}} - \int f\, \nabla_v \frac{\delta \mathcal{G}}{\delta f} \mathrm{d}{\boldsymbol v} \right) - \frac{\delta \mathcal{G}}{\delta p} \left( \nabla \times \frac{\delta \mathcal{F}}{\delta {\boldsymbol B}} - \int f\, \nabla_v \frac{\delta \mathcal{F}}{\delta f} \mathrm{d}{\boldsymbol v} \right) \right) \mathrm{d}{\boldsymbol x},\\
 &\{\mathcal{F}, \mathcal{G} \}_{xv} = \int f  \left(\nabla_x\frac{\delta  \mathcal{F}}{\delta f} \cdot \nabla_v \frac{\delta \mathcal{G}}{\delta f}
 - \nabla_v\frac{\delta  \mathcal{F}}{\delta f} \cdot \nabla_x \frac{\delta \mathcal{G}}{\delta f}\right)  \mathrm{d}{\boldsymbol x}\,  \mathrm{d}{\boldsymbol v}.
\end{aligned}
\end{equation*}
 The energy of the dimensionless hybrid model~\eqref{eq:dimensionlessmodel} is 
\begin{equation}\label{eq:hamiltonianpressu}
\mathcal{E} = \frac{1}{2} \int |{\boldsymbol v}|^2f\, \mathrm{d}{\boldsymbol x}\, \mathrm{d}{\boldsymbol v}
 + \frac{1}{2} \int |{\boldsymbol B}|^2\, \mathrm{d}{\boldsymbol x}\, + \int \frac{p}{\gamma -1}\, \mathrm{d}{\boldsymbol x}=: K_f + K_B + K_p, \quad \gamma \neq 1.
\end{equation}
With energy~\eqref{eq:hamiltonianpressu} and Poisson bracket~\eqref{bra0}, the hybrid model~\eqref{eq:dimensionlessmodel}  can  be written in a Poisson bracket form $\dot{\mathcal{Z}} = \{\mathcal{Z}, \mathcal{E}\}$ with $\mathcal{Z} = (f, {\boldsymbol B}, p)$. The geometric particle-in-cell methods have been constructed based on Poisson brackets in~\cite{GEMPIC,jianyuan, heyang,martin,dual} for the Vlasov--Maxwell equations.

The time discretization is obtained by the Poisson splitting methods~\cite{2020en} based on the decomposition 
of the full bracket~\eqref{bra0} into two parts, which gives 
two sub-steps conserving mass $\mathcal{M} = \int f\, \mathrm{d}{\boldsymbol x}\,\mathrm{d}{\boldsymbol v}$, momentum $\mathcal{P} = \int {\boldsymbol v} f\, \mathrm{d}{\boldsymbol x}\,\mathrm{d}{\boldsymbol v}$, and energy $\mathcal{E}$. The first sub-step called $pvb$, corresponding to the sub-bracket $\{\mathcal{F}, \mathcal{G}\}_{pvb},
$
is 
\begin{equation}
\label{eq:pressure_bv}
\begin{aligned}
& \frac{\partial f}{\partial t} +({\boldsymbol E} + {\boldsymbol v} \times {\boldsymbol B}) \cdot \frac{\partial f}{\partial {\boldsymbol v}} = 0\,,\\
&\frac{\partial {\boldsymbol B}}{\partial t} = - \nabla \times {\boldsymbol E},\quad {\boldsymbol E} = -\frac{\nabla p}{\rho} - \left(\boldsymbol u - \frac{\boldsymbol J}{\rho}  \right) \times {\boldsymbol B},\\
&\frac{\partial p}{\partial t} + \nabla \cdot ({\boldsymbol u}_\mathrm{e}\, p) + (\gamma - 1)\, p\, \nabla \cdot {\boldsymbol u}_\mathrm{e} = 0, \quad {\boldsymbol u}_\mathrm{e} = {\boldsymbol u} - \frac{\boldsymbol J}{\rho}, \quad \gamma \neq 1.
\end{aligned}
\end{equation}
The second sub-step called $xv$, corresponding to sub-bracket $\{\mathcal{F}, \mathcal{G}\}_{xv}$, is 
\begin{equation}\label{eq:pressure_xv}
     \frac{\partial f}{\partial t} = - \boldsymbol v \cdot \nabla_x f. 
\end{equation}

We can derive the equation of ${\boldsymbol J}_f$ or ${\boldsymbol u}$ from the sub-step $pbv$~\eqref{eq:pressure_bv} 
\begin{equation}\label{eq:cancellation} 
\frac{\partial{\boldsymbol J}_f}{\partial t} = -\nabla p + {\boldsymbol J} \times {\boldsymbol B},\,\, \text{or}\,\, \,\frac{\partial{\boldsymbol u}}{\partial t} = -\frac{\nabla p}{\rho} + \frac{{\boldsymbol J} \times {\boldsymbol B}}{\rho},
\quad  {\boldsymbol J}_f = \int {\boldsymbol v} f\,  \mathrm{d}{\boldsymbol v} = \rho{\boldsymbol u},
\end{equation}
i.e., the contribution of the two terms, ${\boldsymbol v} \times {\boldsymbol B}$ and $-{\boldsymbol u} \times {\boldsymbol B}$, cancels to each other in the time evolution equation of the ${\boldsymbol J}_f$ or ${\boldsymbol u}$.
This property needs to be preserved after numerical discretization, otherwise the so-called cancellation problem would arise~\cite{cancellation}.

\section{Discretization}\label{sec:Discretization}
In this section, we present the details of the discretization, which will be used in the discretization of sub-steps in Section~\ref{sec:iso}. Periodic boundary conditions are imposed in space.
Fourier spectral methods~\cite{jieshen} and semi-Lagrangian methods~\cite{semi1} are presented in detail.
For the time discretization of the unknown $a$, we denote its approximation at $n$-th time step, i.e., at $t=n\Delta t$, as $a^n$, where $\Delta t$ is the time step size. For field unknown $a$, $a^{n+\frac{1}{2}} = \frac{1}{2}(a^n+a^{n+1})$.

Here, we define the notations used in the phase-space discretization.
We consider the domain $\Omega = \Omega_x \times \Omega_v$, in which $
\Omega_x = \Pi_{\alpha=1}^3[0, L_\alpha], \Omega_v = \Pi_{\alpha=1}^3 [-v^\alpha_l, v^\alpha_r]
$,
$L_\alpha$, $v^\alpha_l$, and $v^\alpha_r$ are  real numbers, $v^\alpha_l$ and $v^\alpha_r$ are chosen such that it is also reasonable to assume periodic boundary conditions in velocity. We use uniform grids in space and velocity. The grids and duals of grids~\cite{exact1} (in Fourier spaces) in the $\alpha$-th space direction are 
\begin{equation}
\label{eq:frequency}
\begin{aligned}
&\mathbb{G}^{x_\alpha} =\Delta x_\alpha \left[\!\!\left[0,M_\alpha-1 \right]\!\! \right], \quad \Delta x_\alpha = L_\alpha/M_\alpha,\quad \hat{\mathbb{G}}^{x_\alpha} = \frac{2\pi}{L_\alpha} \left[\!\!\left[ -\Big\lfloor\frac{M_\alpha-1}{2}\Big\rfloor,\Big\lfloor\frac{M_\alpha}{2}\Big\rfloor \right]\!\! \right],
\end{aligned}
\end{equation}
where the $M_\alpha$ denotes the number of grid points in $\alpha$-th space direction, the variable implicitly naturally associated with $\mathbb{G}^{x_\alpha}$ (resp.
$\hat{\mathbb{G}}^{x_\alpha}$) is denoted as $x^\alpha$ (resp. $\xi^\alpha$), $x^\alpha_{i_\alpha}$ ($\xi^\alpha_{k_\alpha}$) denotes the $i_\alpha$-th ($k_\alpha$-th) grid in $\mathbb{G}^{x_\alpha}$ ($\hat{\mathbb{G}}^{x_\alpha}$).
The grids in the $\alpha$-th direction of the velocity are 
\begin{equation*}
\begin{aligned}
&\mathbb{G}^{v_\alpha} =\Delta v_\alpha \left[\!\!\left[0,N_\alpha-1 \right]\!\! \right], \quad \Delta v_\alpha = \frac{v^\alpha_r - v^\alpha_l}{N_\alpha},
\end{aligned}
\end{equation*}
where $N_\alpha$ denotes the number of grid points in the $\alpha$-th velocity direction,  and  the variable implicitly naturally associated with $\mathbb{G}^{v_\alpha}$ is denoted as $v^\alpha$, and $v^\alpha_{j_\alpha}$ denotes the $j_\alpha$-th grid in $\mathbb{G}^{v_\alpha}$. We also define $M=M_1M_2M_3$, $N=N_1N_2N_3$, $\Delta {\boldsymbol x} :=\Delta  x_1 \Delta  x_2 \Delta  x_3$, and $\Delta {\boldsymbol v} := \Delta v_1\Delta v_2\Delta v_3$.
We give the variable  ${\mathbf x}$ implicitly associated with the three dimensional grids    $\mathbb{G}^x=\Pi_{\alpha=1}^3\mathbb{G}^{x_\alpha} 
 (\text{resp.}\, \mathbb{G}^v=\Pi_{\alpha=1}^3\mathbb{G}^{v_\alpha})$~\cite{dual}
\begin{equation*}
\begin{aligned}
&{\mathbf x}_{\boldsymbol i} = (x^1_{i_1}, x^2_{i_2}, x^3_{i_3})^\top,\, {\boldsymbol i} = (i_1,i_2,i_3) \in \left[\!\!\left[{\boldsymbol 1},{\mathbf M} \right]\!\! \right]:=\Pi_{\alpha=1}^3\left[\!\!\left[1, M_\alpha \right]\!\! \right],\\
&{\mathbf v}_{\boldsymbol j} = (v^1_{j_1}, v^2_{j_2}, v^3_{j_3})^\top,\, {\boldsymbol j} = (j_1,j_2,j_3) \in \left[\!\!\left[{\boldsymbol 1},{\mathbf N} \right]\!\! \right]:=\Pi_{\alpha=1}^3\left[\!\!\left[1, N_\alpha \right]\!\! \right].
\end{aligned}
\end{equation*} 
Similarly for the frequency, 
we have the variable ${\boldsymbol \xi}$ implicitly associated with the duals of grids $\hat{\mathbb{G}}^x=\Pi_{\alpha=1}^3\hat{\mathbb{G}}^{x_\alpha}$  
\begin{equation*}
    \begin{aligned}
&{\boldsymbol \xi}_{ k} = (\xi^1_{k_1}, \xi^2_{k_2}, \xi^3_{k_3})^\top,\, {\boldsymbol k} = (k_1,k_2,k_3) \in \left[\!\!\left[{\boldsymbol 1},{\mathbf M} \right]\!\! \right]:=\Pi_{\alpha=1}^3\left[\!\!\left[1, M_\alpha \right]\!\! \right].
\end{aligned}
\end{equation*} 

\noindent{\bf Fourier basis, operators, and discretization.}
Here, we define the Fourier basis and corresponding matrices of the differential operators as~\cite{martin}. 
In the $x_\alpha$ direction we have the Fourier basis, 
$
{\boldsymbol \Lambda}^{\alpha} = \left(e^{\mathrm{i}\xi^\alpha_1x_\alpha}, \cdots, e^{\mathrm{i}\xi^\alpha_{M_\alpha}x_\alpha}\right)
$.
We define the three dimensional basis via the tensor products of ${\boldsymbol \Lambda}^{\alpha}$, i.e., we have 
$$\{
{\Lambda}_{\boldsymbol k}({\boldsymbol x}) = \Lambda^{1}_{k_1}\Lambda^{2}_{k_2}\Lambda^{3}_{k_3},\ {\boldsymbol k} = (k_1,k_2,k_3) \in \left[\!\!\left[{\boldsymbol 1},{\mathbf M} \right]\!\! \right]:=\Pi_{\alpha=1}^3\left[\!\!\left[1, M_\alpha \right]\!\! \right]\}
$$
We denote ${\boldsymbol \Lambda} = (\Lambda_k)_{1\le k \le M} \in \mathbb{R}^M$.
For the function $p({\boldsymbol x})$, its Fourier approximation is denoted as $p({\boldsymbol x}) \approx p_h({\boldsymbol x}) = \sum_{k=1}^M \hat{p}_k \Lambda_k =: \hat{\mathbf p}^\top{\boldsymbol \Lambda}$, the values of $p_h$ at grid points $\mathbb{G}^x$ are ${\mathbf p} \in \mathbb{R}^M$ with $p_i = p_h({\mathbf x}_i)$, which is associated with $\hat{\mathbf p}$ via discrete Fourier transform. Similar notations can be given for the $\alpha$-th component of the magnetic field $B_\alpha({\boldsymbol x})$, which is approximated as 
$$
B_\alpha({\boldsymbol x}) \approx B_{\alpha,h}({\boldsymbol x}) = \sum_{k=1}^M \hat{B}_{\alpha,k}\Lambda_k = (\hat{\mathbf B}_\alpha)^\top {\boldsymbol \Lambda},
$$
where $\hat{\mathbf B}_\alpha \in \mathbb{R}^M$ are the Fourier coefficients obtained by discrete Fourier transform for the point values of $B_{\alpha,h}({\boldsymbol x})$ at the grids $\mathbb{G}^x$, i.e., ${\mathbf B}_\alpha \in \mathbb{R}^M$ with its $i$-th component as ${ B}_{\alpha,i} = {B}_{\alpha,h}({\mathbf x}_i)$. 
For the vector valued magnetic field ${\boldsymbol B}$ we have the discretization as 
\begin{equation*}
\begin{aligned}
&{\boldsymbol B}_h^\top = \left( \hat{\mathbf B}_1^\top, \hat{\mathbf B}_2^\top, \hat{\mathbf B}_3^\top \right)
\left(\begin{matrix}
  {\mathbf \Lambda} & 0 & 0 \\
    0 & {\mathbf \Lambda}  & 0 \\
      0 & 0 & {\mathbf \Lambda}
\end{matrix} \right) =: \hat{\mathbf B}^\top \mathbb{\Lambda},
\end{aligned}
\end{equation*}
where $\mathbb{\Lambda} \in \mathbb{R}^{M \times 3}$. The values of magnetic field ${\boldsymbol B}_h$ at grids $\mathbb{G}^x$ are denoted as 
$${\mathbf B} = \left({\mathbf B}_1^\top, {\mathbf B}_2^\top, {\mathbf B}_3^\top\right)^\top \in \mathbb{R}^{3M}.
$$

The matrices corresponding to three differential operators, i.e., gradient $\nabla$, curl $\nabla \times $, and divergence $\nabla \cdot$, in the Fourier basis, are explicitly given in~\cite{martin} via the tensor products of the diagonal one-dimensional derivative matrices. And the spaces spanned by Fourier basis for the scalar and vector valued functions form a de Rham sequence, more details can be found in~\cite{martin}.

 We denote the approximation of the distribution function at uniform grids as 
${\mathbf f} = (f_{{\boldsymbol i}{\boldsymbol j}})$ when the cubic spline based semi-Lagrangian method~\cite{semi1} is adopted. 
The Fourier coefficients $\hat{\mathbf f} = (\hat{f}_{{\boldsymbol k}{\boldsymbol j}})$ are obtained by discrete Fourier transform in space for ${\mathbf f}$. With ${\mathbf f}$, we can take the discrete velocity moments and get ${\boldsymbol \rho}$, ${\mathbf J}_{f,\alpha}$, ${\mathbf u}$, and ${\mathbf u}_e$, with their $i$-th values 
\begin{equation*}
\begin{aligned}
&\rho_{\boldsymbol i} = \sum_{{\boldsymbol j}={\mathbf 1}}^{\mathbf N} f_{{\boldsymbol i}{\boldsymbol j}}\, \Delta {\boldsymbol v}, \quad J_{f,\alpha,{\mathbf i}} = \sum_{{\boldsymbol j}={\mathbf 1}}^{\mathbf N} f_{{\boldsymbol i}{\boldsymbol j}}\,v^\alpha_{j_\alpha}\Delta {\boldsymbol v}, \\ &u_{\alpha,{\boldsymbol i}} = J_{f,\alpha,{\boldsymbol i}}/\rho_{\boldsymbol i}, \quad u_{e,\alpha,{\boldsymbol i}} = u_{\alpha,{\boldsymbol i}} - (\nabla \times {\boldsymbol B}_h)_\alpha({\mathbf x}_{\boldsymbol i})/\rho_{\boldsymbol i},
\end{aligned}
\end{equation*}
which are associated with their Fourier coefficients $\hat{\boldsymbol \rho}$, $\hat{\mathbf J}_{f,\alpha}$, $\hat{\mathbf u}$, and $\hat{\mathbf u}_e$ via discrete Fourier transform. With the Fourier coefficients, we have the trigonometric interpolation polynomials 
$\rho \approx \rho_h = {\boldsymbol \Lambda}^\top \hat{\boldsymbol \rho}$, $J_{f,\alpha} \approx J_{f,\alpha,h} = {\boldsymbol \Lambda}^\top \hat{\mathbf J}_{f,\alpha}$, ${\boldsymbol u}_\alpha \approx {\boldsymbol u}_{\alpha,h} = {\boldsymbol \Lambda}^\top \hat{\mathbf u}_{\alpha}$, and ${\boldsymbol u}_{e,\alpha} \approx {\boldsymbol u}_{e,\alpha,h} = {\boldsymbol \Lambda}^\top \hat{\mathbf u}_{e,\alpha}$
For the vector valued functions, we have the notations
\begin{equation*}
    \begin{aligned}
&{\mathbf J}_{f} = ({\mathbf J}_{f,1}^\top, {\mathbf J}_{f,2}^\top, {\mathbf J}_{f,3}^\top)^\top, \quad {\mathbf J}_{f,{\boldsymbol i}} =  ({J}_{f,1,{\boldsymbol i}}, {J}_{f,2,{\boldsymbol i}}, {J}_{f,3,{\boldsymbol i}})^\top,\\
&{\mathbf u} = ({\mathbf u}_{1}^\top, {\mathbf u}_{2}^\top, {\mathbf u}_{3}^\top)^\top, \quad {\mathbf u}_{\boldsymbol i} =  ({u}_{1,{\boldsymbol i}}, {u}_{2,{\boldsymbol i}}, {u}_{3,{\boldsymbol i}})^\top,\\
&{\mathbf u}_e = ({\mathbf u}_{e,1}^\top, {\mathbf u}_{e,2}^\top, {\mathbf u}_{e,3}^\top)^\top, \quad {\mathbf u}_{e,{\boldsymbol i}} =  ({u}_{e,1,{\boldsymbol i}}, {u}_{e,2,{\boldsymbol i}}, {u}_{e,3,{\boldsymbol i}})^\top.
\end{aligned}
\end{equation*}

\noindent{\bf Computation of the one-dimensional advections.}
In this work, the distribution function is updated via only one dimensional advections with the semi-Lagrangian method~\cite{semi1}, so here we present the details of solving the one dimensional transport equation of the function $f(t, z)$ 
\begin{equation}\label{eq:1dtrans}
 \frac{\partial f}{\partial t} + c \frac{\partial f}{\partial z} = 0, \quad c \in \mathbb{R}.   
\end{equation}
The domain is $[a, b]$ and a periodic boundary condition is assumed. We use a uniform grid $\mathbb{G}^z = \{z_i = a + \Delta z, \Delta z = (b-a)/N, 1 \le i \le  N$\}, and its duals $\hat{\mathbb{G}}^z$ as~\eqref{eq:frequency}, the variable implicitly naturally associated with $\mathbb{G}^{z}$ (resp.
$\hat{\mathbb{G}}^{z}$) is denoted as $z$ (resp. $\beta$), 
The discretization of $f$, i.e., ${\mathbf f} =(f_i),1\le i \le N$ are a function defined on $\mathbb{G}^z$. 
When we use Fourier spectral method, ${\mathbf f}$ is updated in time as
\begin{equation}\label{eq:velospec}
{\mathbf f}^{n+1} =\mathcal{F}_z^{-1} e^{-\mathrm{i}c\Delta t{\beta}}\mathcal{F}_z\,{\mathbf f}^n,
\end{equation}
where $\mathcal{F}_z$ is the matrix of discrete Fourier transform,
$$
\mathcal{F}_z: \mathbb{C}^{\mathbb{G}^z} \rightarrow \mathbb{C}^{\hat{\mathbb{G}}^z}: {\mathbf f} \rightarrow \hat{\mathbf f} = \sum_{z \in \mathbb{G}} f_z e^{-\mathrm{i}z\beta}
$$
$\mathcal{F}_z\,{\mathbf f}$ gives the Fourier coefficients of ${\mathbf f}$, i.e., $\hat{\mathbf f}:=(\hat{f}_k)$ defined on $\hat{\mathbb{G}}^z$.
\begin{lemma}\label{lemmapu}
For the translation~\eqref{eq:velospec} conducted with Fourier spectral method, we have 
$
\sum_{i=1}^N f^{n+1}_i = \sum_{i=1}^N f^{n}_i 
$.
\end{lemma}
\begin{proof}
Based on the definition of discrete Fourier transform~\cite{jieshen}, we have 
$$\sum_{i=1}^N f^{n+1}_i = \hat{f}_{\lfloor\frac{N-1}{2}\rfloor+1}^{n+1},\quad  \sum_{i=1}^N f^{n}_i = \hat{f}_{\lfloor\frac{N-1}{2}\rfloor+1}^{n}.$$
As $\beta_{\lfloor\frac{N-1}{2}\rfloor+1}=0$, $\hat{f}_{\lfloor\frac{N-1}{2}\rfloor+1}^{n}=\hat{f}_{\lfloor\frac{N-1}{2}\rfloor+1}^{n+1}$ and $\sum_{i=1}^N f^{n+1}_i = \sum_{i=1}^N f^{n}_i $.
\end{proof}
When we use the cubic spline based semi-Lagrangian method~\cite{semi1}, firstly we reconstruct a continuous function $\mathcal{I}f^n(z) := \sum_{i=1}^N e_i C(z-z_j)$ using $(f^n_i)$, where $C(z)$ is the cubic B-spline, 
\begin{equation*}
C(z) = \frac{1}{6}
\left\{
    \begin{alignedat}{2}
    &\left(2-{|z|}/{\Delta z}\right)^3 \quad & \quad &\text{if} \quad \Delta z \le |z| < 2\Delta z\\
    & 4-6(|z|/\Delta z)^2 + 3(|z|/\Delta z)^3 \quad & \quad &\text{if} \quad 0 \le |z| < \Delta z\\
    & 0  & \quad &\text{else}
    \end{alignedat}
\right.
\end{equation*}
and $(e_i)$ is obtained by a fast solver given in~\cite{bookvlasov} when periodic boundary condition is imposed. When the support of $C(z-z_j)$ extends beyond the domain, we handle it using the periodicity assumption.
Then the distribution function is updated as 
\begin{equation}\label{eq:cubicsemi}
f^{n+1}_i = \mathcal{I}f^n(z_i-c\Delta t).
\end{equation}
The cubic splines on the uniform grids $\{z_i\},i\in \mathbb{Z}$ of the velocity domain $\mathbb{R}$ have the following properties about its moments~\cite{udovicic2009calculation}. 
\begin{pro}\label{pro:cubicmoment} 
For the cubic spline $C(z)$, we have
$$
\int_{\mathbb{R}} z^s\, C(z)\, \mathrm{d}{z} = \sum_{i \in \mathbb{Z}} (z_i+r)^s\, C(z_i+r)\, \Delta z, \quad \forall r \in \mathbb{R}, s = 0, 1, 2.
$$
\end{pro}
Based the Proposition~\ref{pro:cubicmoment}, we have the following result about change of variables.
\begin{pro}\label{pro:secondpro}
$$
\sum_{i \in \mathbb{Z}} z_i^s\, C(z_i+r)\, \Delta z = \sum_{i \in \mathbb{Z}} (z_i-r)^s\, C(z_i)\, \Delta z, \quad \forall r \in \mathbb{R}, s = 0, 1, 2.
$$
\begin{proof}
We here focus on the case with $s=2$, the cases with $s=0,1$ can be proved similarly.
\begin{equation*}
    \begin{aligned}
&\sum_i z_i^2\, C(z_i+r)\, \Delta z = \sum_i (z_i+r-r)^2\, C(z_i+r) \Delta z\\
& = \sum_i (z_i+r)^2\, C(z_i+r)\, \Delta z + \sum_i r^2\, C(z_i+r)\, \Delta z - \sum_i 2r(z_i+r)\, C(z_i+r)\, \Delta z\\
& = \sum_i z_i^2\, C(z_i)\, \Delta z + \sum_i r^2\, C(z_i)\, \Delta z - \sum_i 2\,rz_i C(z_i)\, \Delta z\\
& = \sum_i (z_i-r)^2\, C(z_i) \Delta z,
 \end{aligned}
\end{equation*}
where the third equality holds because of the Proposition~\ref{pro:cubicmoment}.
\end{proof}
\end{pro}
\begin{pro}\label{pro:advecmomen} For the cubic spline based semi-Lagrangian scheme~\eqref{eq:cubicsemi} with an infinite domain $\mathbb{R}$ and uniform grids $\{z_i\}, i \in \mathbb{Z}$, we have
$$
\sum_{i \in \mathbb{Z}} f^{n+1}_i\, z_i^s = \sum_{i \in \mathbb{Z}} f_i^n\,(z_i + c\Delta t)^s, \quad s = 0, 1, 2.
$$
\end{pro}
\begin{proof}
For the cubic spline interpolation of the $(f^n_i)$, $\mathcal{I}f^n(z) = \sum_{j \in \mathbb{Z}} e_j N(z - z_j)$, we have $f_i^n = \mathcal{I}f^n(z_i)$, and $f_i^{n+1} = \mathcal{I}f^n(z_i - c\Delta t)$.
Then we have  
\begin{equation*}
\begin{aligned}
    \sum_{i \in \mathbb{Z}} f^{n+1}_i\, z_i^s &=  \sum_{i \in \mathbb{Z}}\sum_{j \in \mathbb{Z}} e_j\, C(z_i - z_j - c\Delta t)\, z_i^s= \sum_{i \in \mathbb{Z}}\sum_{j \in \mathbb{Z}} e_j\, C(z_i - z_j)\, (z_i+ c\Delta t)^s \\
    &= \sum_{i \in \mathbb{Z}} f_i^n\,(z_i + c\Delta t)^s,
\end{aligned}
\end{equation*}
where for each $j$, we have used the Proposition~\ref{pro:secondpro}.
\end{proof}
Next we give three identities of the scheme~\eqref{eq:cubicsemi}, which will be used in the following for proving the conservation properties of the schemes. 
\begin{lemma}\label{pro:semiidentity}
    For the cubic spline based semi-Lagrangian method~\eqref{eq:cubicsemi} of the equation~\eqref{eq:1dtrans} with an infinite domain $\mathbb{R}$ and uniform grids $\{z_i\}, i \in \mathbb{Z}$, we have the following identities
    \begin{equation*}
    \begin{aligned}
    &\sum_{i \in \mathbb{Z}} f^{n+1}_i = \sum_{i \in \mathbb{Z}} f^{n}_i,\quad \sum_{i \in \mathbb{Z}} z_i\,f^{n+1}_i = \sum_{i \in \mathbb{Z}} z_i\,f^{n}_i + c\,\Delta t \sum_{i \in \mathbb{Z}} f^n_i,\\
    &\frac{1}{2}\sum_{i \in \mathbb{Z}} z_i^2\,f^{n+1}_i = \frac{1}{2}\sum_{i \in \mathbb{Z}} z_i^2\,f^{n}_i + c\,\Delta t \left( \sum_{i \in \mathbb{Z}} z_i\,f^{n}_i + \frac{c\Delta t}{2}  \sum_{i \in \mathbb{Z}} f^{n}_i \right).
    \end{aligned}
    \end{equation*}
    \end{lemma}
    \begin{proof}
    The first two equalities have already been proved in the Proposition 14 and 15 in~\cite{bookvlasov}. Here we focus on the third equality, which can be proved with Proposition~\ref{pro:advecmomen} and has been mentioned in~\cite{bookvlasov}.
    For the cubic spline interpolation of the $(f^n_i)$, $\mathcal{I}f^n(z) = \sum_{j} e_j C(z - z_j)$, we have $f_i^n = \mathcal{I}f^nf(z_i)$, and $f_i^{n+1} = \mathcal{I}f^n(z_i - c\Delta t)$. With the Proposition~\ref{pro:advecmomen}, 
    we have 
    \begin{equation*}
    \begin{aligned}
& \frac{1}{2}\sum_{i \in \mathbb{Z}} z_i^2\,f^{n+1}_i = \frac{1}{2}\sum_{i \in \mathbb{Z}} (z_i + c\Delta t) ^2 f^n_i \\
& = \frac{1}{2}\sum_{i \in \mathbb{Z}} z_i^2\, f_i^n + \frac{1}{2}\sum_{i \in \mathbb{Z}}  c^2\Delta t^2 f^n_i + c\,\Delta t  \sum_{i \in \mathbb{Z}} z_i\, f^n_i.
    \end{aligned}  
    \end{equation*}
    Thus we have finished the proof.
    \end{proof}

\begin{remark}\label{rm:truncate}
    In practical simulations, the infinite velocity domain is truncated to a sufficient large finite domain, such that the distribution functions are very small near the boundaries, which makes the Proposition~\ref{pro:advecmomen}-\ref{pro:semiidentity} hold with round-off errors for the first and second order moments~\cite{bookvlasov}.
\end{remark}

    \section{Sub-steps}\label{sec:iso}
Here, we present the discretization for sub-steps $pvb$ and $xv$. 
After full discretization, we have the following discrete mass $\mathcal{M}_h$, $\alpha$-th component of momentum $\mathcal{P}_{h,\alpha}$, and energy $\mathcal{E}_h$, the conservation of which will be considered for each sub-step,
\begin{equation}
\label{eq:discretequantities}
\begin{aligned}
\mathcal{M}_h &=  \sum_{{\boldsymbol i}{\boldsymbol j}}  f_{{\boldsymbol i}{\boldsymbol j}} \,\Delta {\boldsymbol x}\,\Delta {\boldsymbol v}, \quad 
 \mathcal{P}_{h,\alpha} = \sum_{{\boldsymbol i}{\boldsymbol j}}v^{\alpha}_{j_\alpha}f_{{\boldsymbol i}{\boldsymbol j}}\, \Delta {\boldsymbol x}\,\Delta {\boldsymbol v}, \\
\mathcal{E}_h &= \frac{1}{2}\sum_{{\boldsymbol i}{\boldsymbol j}} \left( |v^1_{j_1}|^2 + |v^2_{j_2}|^2 + |v^3_{j_3}|^2 \right)f_{\boldsymbol {ij}}\, \Delta {\boldsymbol x}\,\Delta {\boldsymbol v} \\
& + \frac{1}{2} \sum_{{\boldsymbol i}} \left({ B}_{1,{\boldsymbol i}}^2 +{ B}_{2,{\boldsymbol i}}^2 + { B}_{3,{\boldsymbol i}}^2 \right)\, \Delta {\boldsymbol x}  + \frac{1}{\gamma-1} \sum_{{\boldsymbol i}} p_{\boldsymbol i}\,  \Delta {\boldsymbol x} =: K_{f,h} + K_{B,h} + K_{p,h}.
\end{aligned}
\end{equation}
As mentioned in Remark~\ref{rm:truncate}, in this section we work with a periodic boundary condition in space and an infinite velocity domain, and in the practical simulations we have to truncate the velocity domain into a sufficiently large finite velocity domain, which makes the conservation properties of the momentum and energy we prove in this section hold with round off errors~\cite{bookvlasov}.

\subsection{Sub-step $pvb$}
The equation of this sub-step is 
\begin{equation}
\label{eq:discretization_isothermal_pbv}
\begin{aligned}
& \frac{\partial f}{\partial t} +({\boldsymbol E} + {\boldsymbol v} \times {\boldsymbol B}) \cdot \frac{\partial f}{\partial {\boldsymbol v}} = 0\,,\\
&\frac{\partial {\boldsymbol B}}{\partial t} = - \nabla \times {\boldsymbol E},\quad {\boldsymbol E} = -\frac{\nabla p}{\rho} - \left(\boldsymbol u - \frac{\nabla\times{\boldsymbol B}}{\rho}  \right) \times {\boldsymbol B},\\
&\frac{\partial p}{\partial t} + \nabla \cdot ({\boldsymbol u}_\mathrm{e}\, p) + (\gamma - 1)\, p\, \nabla \cdot {\boldsymbol u}_\mathrm{e} = 0, \quad {\boldsymbol u}_\mathrm{e} = {\boldsymbol u} - \frac{\nabla \times {\boldsymbol B}}{\rho}, \quad \gamma \neq 1.
\end{aligned}
\end{equation}

\noindent{\bf Time discretization.}
 The following modified mid-point rule is used for the time discretization between the time interval $[t^n, t^{n+1}]$ 
\begin{equation}\label{eq:bvubarf}
  \begin{aligned}
&\frac{\partial f}{\partial t} +\left({\boldsymbol v} - \bar{\boldsymbol u} + \frac{\nabla \times {\boldsymbol B}^{n+\frac{1}{2}}}{\rho}  \right) \times {\boldsymbol B}^{n+\frac{1}{2}} \cdot \frac{\partial f}{\partial {\boldsymbol v}}  - \frac{{\nabla p}^{n+\frac{1}{2}}}{\rho} \cdot \frac{\partial f}{\partial {\boldsymbol v}} = 0,\\
&\frac{{\boldsymbol B}^{n+1} - {\boldsymbol B}^n}{\Delta t} =  \nabla \times \left( \bar{\boldsymbol u} \times {\boldsymbol B}^{n+\frac{1}{2}}\right) - \nabla \times \left( \frac{\nabla \times {\boldsymbol B}^{n+\frac{1}{2}}}{\rho} \times {\boldsymbol B}^{n+\frac{1}{2}}\right) + \nabla \times \frac{{\nabla p}^{n+\frac{1}{2}}}{\rho},\\
&\frac{{p}^{n+1} - {p}^n}{\Delta t}  + \nabla \cdot \left(p^{n+\frac{1}{2}}\, \bar{\boldsymbol u}_e\right) + (\gamma - 1) p^{n+\frac{1}{2}}\, \nabla \cdot \bar{\boldsymbol u}_e  = 0, \quad \bar{\boldsymbol u}_e = \bar{\boldsymbol u} - \frac{\nabla \times {\boldsymbol B}^{n+\frac{1}{2}}}{\rho},
\end{aligned}
\end{equation}
in which the $f$ is updated with
as 
\begin{equation}\label{eq:semitraceback}
f(t^{n+1}, {\boldsymbol x}, {\boldsymbol v}) = f(t^n, {\boldsymbol x}, {\boldsymbol v}^n  ),
\end{equation}
where ${\boldsymbol v}^n$ is the solution of the following characteristics at time $t^{n}$ with ${\boldsymbol v}(t^{n+1}) = {\boldsymbol v}$,
 \begin{equation*}
\begin{aligned}
\dot{\boldsymbol v} & = \left({\boldsymbol v} -\bar{\boldsymbol u} + \frac{\nabla \times {\boldsymbol B}^{n+\frac{1}{2}}}{\rho} \right) \times {\boldsymbol B}^{n+\frac{1}{2}} - \frac{{\nabla p}^{n+\frac{1}{2}}}{\rho}.
\end{aligned}
\end{equation*}
 The density $\rho$ is not changing in time, so we omit the time index. The $\bar{\boldsymbol u}$ is fixed by the following condition, 
\begin{equation}\label{eq:u_int}
     \bar{\boldsymbol u} = \frac{1}{\Delta t} \int_{t^n}^{t^{n+1}}{\boldsymbol u}(t)\, \mathrm{d}{t},
 \end{equation}
where ${\boldsymbol u}(t)$ is the solution of equation 
\begin{equation}\label{eq:bvubarf3}
\frac{\partial {\boldsymbol u}}{\partial t} -\left({\boldsymbol u} - \bar{\boldsymbol u} + \frac{\nabla \times {\boldsymbol B}^{n+\frac{1}{2}}}{\rho}  \right) \times {\boldsymbol B}^{n+\frac{1}{2}} +\frac{{\nabla p}^{n+\frac{1}{2}}}{\rho} = 0,  \quad {\boldsymbol u}(t^n) = {\boldsymbol u}^n,
\end{equation}
which is obtained by the moment of the Vlasov equation in~\eqref{eq:bvubarf}. Thus $\bar{\boldsymbol u}$ is the time average of ${\boldsymbol u}(t)$ between time $[t^n, t^{n+1}]$. 

\noindent{\bf The explicit formula of $\bar{\boldsymbol u}$.}
By splitting $\frac{{\nabla p}^{n+\frac{1}{2}}}{\rho}$ into the parallel and perpendicular parts of the magnetic field, i.e., when ${\boldsymbol B}^{n+\frac{1}{2}}\neq 0$,
\begin{equation}\label{eq:pressuredecom}
\begin{aligned}
&\frac{{\nabla p}^{n+\frac{1}{2}}}{\rho} = {\boldsymbol q} \times {\boldsymbol B}^{n+\frac{1}{2}} + \left(\frac{{\nabla p}^{n+\frac{1}{2}}}{\rho}\right)_\parallel,\quad {\boldsymbol q} =\frac{ {\boldsymbol B}^{n+\frac{1}{2}}}{ |{\boldsymbol B}^{n+\frac{1}{2}}|^2} \times \frac{{\nabla p}^{n+\frac{1}{2}}}{\rho}, \\ &\left(\frac{{\nabla p}^{n+\frac{1}{2}}}{\rho}\right)_\parallel = \frac{{\nabla p}^{n+\frac{1}{2}} \cdot {\boldsymbol B}^{n+\frac{1}{2}}  }{\rho} \frac{{\boldsymbol B}^{n+\frac{1}{2}}}{|{\boldsymbol B}^{n+\frac{1}{2}}|^2}, 
\end{aligned}
\end{equation}
and when ${\boldsymbol B}^{n+\frac{1}{2}} = 0$, we set ${\boldsymbol q}=0$ and regard   $\frac{{\nabla p}^{n+\frac{1}{2}}}{\rho}$ parallel to ${\boldsymbol B}^{n+\frac{1}{2}}$, 
we have the solution of~\eqref{eq:bvubarf3} 
\begin{equation*}
    \begin{aligned}
        {\boldsymbol u}(t) &= e^{(t-t^n)\hat{\boldsymbol B}^{n+\frac{1}{2}}}\left({\boldsymbol u}^n - \bar{\boldsymbol u} + \frac{\nabla \times {\boldsymbol B}^{n+\frac{1}{2}}}{\rho}  - {\boldsymbol q}   \right)\\
        &+ \bar{\boldsymbol u}  - \frac{\nabla \times {\boldsymbol B}^{n+\frac{1}{2}}}{\rho} +  {\boldsymbol q} - (t-t^n)\left(\frac{{\nabla p}^{n+\frac{1}{2}}}{\rho}\right)_\parallel,
    \end{aligned}
\end{equation*}
where $\hat{\boldsymbol B} {\boldsymbol v} := {\boldsymbol v} \times {\boldsymbol B}$.
 With the explicit formula of $e^{t\hat{\boldsymbol B}^{n+\frac{1}{2}}}$~\cite{he2015volume} 
 $$
e^{t\hat{\boldsymbol B}^{n+\frac{1}{2}}} = \mathbb{I}_3 + \frac{\sin(t|{\boldsymbol B}^{n+\frac{1}{2}}|)}{|{\boldsymbol B}^{n+\frac{1}{2}}|}\hat{\boldsymbol B}^{n+\frac{1}{2}} + 2 \frac{\sin^2(\frac{t}{2}|{\boldsymbol B}^{n+\frac{1}{2}}|)}{|{\boldsymbol B}^{n+\frac{1}{2}}|^2}\hat{\boldsymbol B}^{n+\frac{1}{2}}\hat{\boldsymbol B}^{n+\frac{1}{2}} 
 $$
 and 
by directly calculating $ \int_{t^n}^{t^{n+1}}{\boldsymbol u}(t)\, \mathrm{d}{t}$, we have the following explicit formula of $\bar{\boldsymbol u}$,
\begin{equation}\label{eq:explicitbaru}
    \begin{aligned}
\bar{\boldsymbol u} &= {\boldsymbol u}^n +  \frac{\nabla \times  {\boldsymbol B}^{n+\frac{1}{2}}}{\rho} -  {\boldsymbol q} + {\boldsymbol M}^{-1} \left( -\frac{\nabla \times  {\boldsymbol B}^{n+\frac{1}{2}}}{\rho}+  {\boldsymbol q} - \frac{\Delta t}{2} \left(\frac{{\nabla p}^{n+\frac{1}{2}}}{\rho}\right)_\parallel  \right),
 \end{aligned}
\end{equation}
where ${\boldsymbol M} = \frac{1}{\Delta t}\int_{0}^{\Delta t}e^{\tau\hat{\boldsymbol B}^{n+\frac{1}{2}}} \mathrm{d}{\tau} $. 
\begin{pro}
    The matrix ${\boldsymbol M}$ is invertible if $\Delta t |{\boldsymbol B}^{n+\frac{1}{2}}| \neq 2k\pi, k \in \mathbb{Z}, k \ge 1$.
\end{pro}
\begin{proof}
We denote the eigenvalue of $\hat{\boldsymbol B}^{n+\frac{1}{2}}$ as $\lambda$, which is $0$, $\mathrm{i}|{{\boldsymbol B}^{n+\frac{1}{2}}}|$ or $-\mathrm{i}|{{\boldsymbol B}^{n+\frac{1}{2}}}|$.
Then the eigenvalue of ${\boldsymbol M}$ is 1 or 
$
\frac{\sin \theta}{\theta} \pm 2 \sin^2\left(\frac{\theta}{2}\right)
$, where $\theta = \Delta t |{\boldsymbol B}^{n+\frac{1}{2}}|$. Then we know that when $\theta \neq 2k\pi, k \in \mathbb{Z}, k \ge 1$, all the eigenvalues of matrix ${\boldsymbol M}$ are non-zero, and matrix ${\boldsymbol M}$ is invertible. 
\end{proof}
%which is invertible and has ${\boldsymbol B}^{n+\frac{1}{2}}$ as eigenvector corresponding to eigenvalue 1.

As the time integral average approximates the mid-point value with the error of $\mathcal{O}(\Delta t^2)$,  by comparing it with the standard mid-point rule, we know that the modified mid-point rule~\eqref{eq:bvubarf} is second order in time.

\noindent{\bf Reversibility.} Then we consider the reversibility~\cite{HLW} of the modified mid-point rule~\eqref{eq:bvubarf}. 
\begin{theorem}\label{thm:reversible1}
    The scheme~\eqref{eq:bvubarf} is reversible.
\end{theorem}
\begin{proof}
    By the definition of reversible schemes, we need to prove when we exchange the unknowns at $t^n$ with $t^{n+1}$, and set $\Delta t = -\Delta t$, the same scheme as~\eqref{eq:bvubarf} is obtained. 

    For the scheme of the magnetic field and pressure in~\eqref{eq:bvubarf}, we can see we only need to prove $\bar{\boldsymbol u}$ is not changed when we exchange ${\boldsymbol u}^n$ with ${\boldsymbol u}^{n+1}$ and set $\Delta t = -\Delta t$, which is true due to the equation~\eqref{eq:bvubarf3} satisfied by ${\boldsymbol u}$ is reversible. Also the scheme of $f$ is reversible, due to reversibility of the Vlasov equation in~\eqref{eq:bvubarf}.
\end{proof}

\begin{theorem}
When we assume a periodic boundary condition in space, and an infinite velocity domain,
    mass, momentum, and energy are conserved by the time discretization~\eqref{eq:bvubarf}. Also the cancellation problem~\eqref{eq:cancellation} is overcome.
\end{theorem}
\begin{proof}
    The mass conservation is easy to prove by directly integrating the Vlasov equation in~\eqref{eq:bvubarf} in phase-space with the periodic boundary condition.

    Multiplying ${\boldsymbol v}$ on both sides of the Vlasov equation in~\eqref{eq:bvubarf} and then integrating about ${\boldsymbol v}$ and $t$ gives 
    \begin{equation}
    \label{eq:jfsemi}
\begin{aligned}
      {\boldsymbol J}_f^{n+1} - {\boldsymbol J}_f^{n} &=   \int_{t^n}^{t^{n+1}} {\boldsymbol J}_f(t) \times {\boldsymbol B}^{n+\frac{1}{2}}\,\mathrm{d}{t} - \Delta t\, \rho \bar{\boldsymbol u}\times {\boldsymbol B}^{n+\frac{1}{2}} \\
      & + \Delta t\, {\nabla \times  {\boldsymbol B}^{n+\frac{1}{2}}} \times {\boldsymbol B}^{n+\frac{1}{2}} - \Delta t\, \nabla p^{n+\frac{1}{2}},
    \end{aligned}
    \end{equation}
    where the first two terms cancel out due to the condition~\eqref{eq:u_int}, and thus the cancellation problem~\eqref{eq:cancellation} is overcome. By integrating the above identity~\eqref{eq:jfsemi} about ${\boldsymbol{ x}}$, we get 
    \begin{equation}
        \begin{aligned}
    \mathcal{P}^{n+1} - \mathcal{P}^n &= \int  \Delta t\, {\nabla \times  {\boldsymbol B}^{n+\frac{1}{2}}} \times {\boldsymbol B}^{n+\frac{1}{2}} - \Delta t\, \nabla p^{n+\frac{1}{2}}\, \mathrm{d}{\boldsymbol x}\\
    &  =\int \Delta t\, \nabla \cdot \left({\boldsymbol B}^{n+\frac{1}{2}}{\boldsymbol B}^{n+\frac{1}{2}} - \frac{|{\boldsymbol B}^{n+\frac{1}{2}}|^2}{2}\mathbb{I}_3\right) - \Delta t\, \nabla p^{n+\frac{1}{2}}\, \mathrm{d}{\boldsymbol x} = 0,
    \end{aligned}
    \end{equation}
    where in the second last equality we have used the the condition $\nabla \cdot {\boldsymbol B}^{n+\frac{1}{2}} = 0$. Then we have proved the momentum conservation.

    Multiplying $\frac{|{\boldsymbol v}|^2}{2}$ on both sides of the Vlasov equation in~\eqref{eq:bvubarf} and then integrating in the phase-space and about $t$ gives
    \begin{equation}
    \label{eq:kfen}
        \begin{aligned}
K_f^{n+1} - K_f^n &= \int - \bar{\boldsymbol u} \times {\boldsymbol B}^{n+\frac{1}{2}} \cdot {\boldsymbol J}_f\,\mathrm{d}{\boldsymbol x}\, \mathrm{d}t \underbrace{- \int \frac{\nabla {p}^{n+\frac{1}{2}}}{\rho} \cdot {\boldsymbol J}_f\, \mathrm{d}{\boldsymbol x}\, \mathrm{d}t}_{a}\\
&+ \underbrace{\int \frac{\nabla \times {\boldsymbol B}^{n+\frac{1}{2}}}{\rho} \times {\boldsymbol B}^{n+\frac{1}{2}} \cdot {\boldsymbol J}_f\, \mathrm{d}{\boldsymbol x}\, \mathrm{d}t}_{b},
     \end{aligned}
    \end{equation}
    where the first term vanishes due to the condition~\eqref{eq:u_int}.
Multiplying ${\boldsymbol B}^{n+\frac{1}{2}}$ on both sides of the equation of magnetic field in~\eqref{eq:bvubarf} and then integrating about ${\boldsymbol x}$ gives
\begin{equation*}
\begin{aligned}
K_B^{n+1} - K_B^n & = - \Delta t \int {\boldsymbol B}^{n+\frac{1}{2}} \cdot  \nabla \times \left( \frac{\nabla \times {\boldsymbol B}^{n+\frac{1}{2}}}{\rho} \times {\boldsymbol B}^{n+\frac{1}{2}}\right) \mathrm{d}{\boldsymbol x}\\ &
+ \underbrace{\Delta t \int {\boldsymbol B}^{n+\frac{1}{2}} \cdot   \nabla \times \left( \bar{\boldsymbol u} \times {\boldsymbol B}^{n+\frac{1}{2}}\right) \mathrm{d}{\boldsymbol x}}_{b} + \underbrace{\Delta t\int {\boldsymbol B}^{n+\frac{1}{2}} \cdot  \nabla \times \frac{{\nabla p}^{n+\frac{1}{2}}}{\rho} \mathrm{d}{\boldsymbol x}}_{c},
\end{aligned}
\end{equation*}
where the first term vanishes via integration by parts.
Integrating the pressure equation in~\eqref{eq:bvubarf} about ${\boldsymbol x}$ gives
\begin{equation*}
    \begin{aligned}
K_p^{n+1} - K_p^n &= \frac{-\Delta t}{\gamma -1} \int \nabla \cdot (p^{n+\frac{1}{2}}\, \bar{\boldsymbol u}_e) \,\mathrm{d}{\boldsymbol x} \underbrace{- \Delta t \int p^{n+\frac{1}{2}} \nabla \cdot \bar{\boldsymbol u}\, \mathrm{d}{\boldsymbol x}}_{a}\\
&+ \underbrace{\Delta t \int p^{n+\frac{1}{2}} \nabla \cdot \frac{\nabla \times {\boldsymbol B}^{n+\frac{1}{2}}}{\rho}\, \mathrm{d}{\boldsymbol x}}_{c},
 \end{aligned}
\end{equation*}
where the first term vanishes via integration by parts.  The sum of the remaining terms with the same label cancel out via integration by parts and the condition~\eqref{eq:u_int}, and the energy conservation is proved. 
\end{proof}

\noindent{\bf Exact splitting.}
For the resolution of the distribution function, if we directly use semi-Lagrangian method~\cite{semi1} to solve~\eqref{eq:semitraceback}, costly three dimensional interpolations are needed.  According to the theory in~\cite{exact2}, exact splitting that avoids high dimensional interpolations can be constructed for~\eqref{eq:semitraceback}. In the following, we give an explicit exact splitting for~\eqref{eq:semitraceback} with the help of the decomposition of the pressure term~\eqref{eq:pressuredecom} and the exact splitting for rotations~\cite{exact1,exact2}. With the exact splitting presented below, only one dimensional advections and interpolations are needed, and are thus efficient.

We firstly have a closer look at the characteristics,
\begin{equation}
\label{eq:characteristics}
\begin{aligned}
\dot{\boldsymbol v} & = \left({\boldsymbol v} -\bar{\boldsymbol u} - {\boldsymbol q} + \frac{\nabla \times {\boldsymbol B}^{n+\frac{1}{2}}}{\rho} \right) \times {\boldsymbol B}^{n+\frac{1}{2}} - \left(\frac{{\nabla p}^{n+\frac{1}{2}}}{\rho}\right)_\parallel \\
& =: \hat{\boldsymbol B}^{n+\frac{1}{2}}\left({\boldsymbol v} - \bar{\boldsymbol u} - {\boldsymbol q} + \frac{\nabla \times {\boldsymbol B}^{n+\frac{1}{2}}}{\rho}\right) - \left(\frac{{\nabla p}^{n+\frac{1}{2}}}{\rho}\right)_\parallel,
\end{aligned}
\end{equation}
the solution of which with 
{\small{${\boldsymbol v}(t^n) = {\boldsymbol v}^n$ is 
$$
{\boldsymbol v}(t) = e^{(t-t^n)\hat{\boldsymbol B}^{n+\frac{1}{2}}}\left({\boldsymbol v}^n  - \bar{\boldsymbol u} - {\boldsymbol q} + \frac{\nabla \times {\boldsymbol B}^{n+\frac{1}{2}}}{\rho} \right) + \bar{\boldsymbol u} + {\boldsymbol q} - \frac{\nabla \times {\boldsymbol B}^{n+\frac{1}{2}}}{\rho} - (t-t^n)\left(\frac{{\nabla p}^{n+\frac{1}{2}}}{\rho}\right)_\parallel.  
$$}}
The rotation matrix $e^{(t-t^n)\hat{\boldsymbol B}^{n+\frac{1}{2}}}$ can be decomposed as the product of four shear matrices~\cite{chenbao, welling,exact2,exact1}. 
Correspondingly, the distribution function can be solved via several one dimensional translations.
Specifically, the distribution function can be updated as follows,
\begin{itemize}
\item The 1st step, one dimensional advections in three velocity directions.
\begin{equation}\label{eq:firsts}
f^n({\boldsymbol v}) \rightarrow f^n\left({\boldsymbol v} + \Delta t \left(\frac{\nabla p^{n+\frac{1}{2}}}{\rho} \right)_\parallel \right) =: f^{n}_1({\boldsymbol v})\end{equation}.
\item The 2nd step, one dimensional advections in three velocity directions.\begin{equation}\label{eq:seconds}
f^{n}_1({\boldsymbol v}) \rightarrow f^{n}_1\left({\boldsymbol v} + \bar{\boldsymbol u} + {\boldsymbol q}- \frac{\nabla \times {\boldsymbol B}^{n+\frac{1}{2}}}{\rho}\right) =: f^{n}_2({\boldsymbol v}).
\end{equation}
\item The 3rd step. Solve the equation $\frac{\partial f}{\partial t} + ({\boldsymbol v} \times {\boldsymbol B}^{n+\frac{1}{2}}) \cdot \frac{\partial f}{\partial {\boldsymbol v}} = 0$ by exact splittings proposed in~\cite{exact2,exact3,exact1}, i.e,
{\small{
\begin{equation}\label{eq:exact3d}
f^{n}_3({\boldsymbol v}) = e^{-\Delta t\hat{\boldsymbol B}^{n+\frac{1}{2}}{\boldsymbol v}\cdot\nabla_{v}}f^{n}_2({\boldsymbol v}) = e^{\Delta ty^{(l)}\cdot {\boldsymbol v}\,\partial_{v_s}}\,\Pi_{k,k\neq s}e^{\Delta ty^{(k)}\cdot {\boldsymbol v}\,\partial_{v_k}}\, e^{\Delta ty^{(r)}\cdot {\boldsymbol v}\,\partial_{v_s}} f^{n}_2({\boldsymbol v}),
\end{equation}
}}
where $s$ is any fixed number in $\{1,2,3 \}$, $y^{(l)}_s = y^{(r)}_s = 0$, for $\forall k \neq s$,  $y^{(k)}_k = 0$, $k \in \{1,2,3\}$, and $y^{(l)}, y^{(r)}, y^{(k)} \in \mathbb{R}^3$ can be obtained by a iteration method given in~\cite{exact1}. $f^n_3$ is obtained from $f^n_2$ by four steps, in each step only one dimensional advections are needed.
\item The 4th step, one dimensional advections in three velocity directions.
\begin{equation}\label{eq:finals}
f^{n}_3({\boldsymbol v}) \rightarrow f^{n}_3\left({\boldsymbol v} -  \bar{\boldsymbol u} -{\boldsymbol q} + \frac{\nabla \times {\boldsymbol B}^{n+\frac{1}{2}}}{\rho}\right) =: f^{n+1}({\boldsymbol v},t).
\end{equation}
\end{itemize}  
\begin{remark}\label{rm:commutable}
When we update the distribution function via one dimensional advections, 
we can also compute the second, third and fourth procedures~\eqref{eq:seconds}-\eqref{eq:finals}, and then compute
the first procedure~\eqref{eq:firsts}, i.e., \eqref{eq:firsts} is commutative with \eqref{eq:seconds}-\eqref{eq:finals}.  The reason can be seen from the solution of the characteristics~\eqref{eq:characteristics}. The term $- \Delta t\left({{\nabla p}^{n+\frac{1}{2}}}/{\rho}\right)_\parallel$ is perpendicular to the magnetic field, and we have the following equivalent solution of the characteristics~\eqref{eq:characteristics}, i.e., ${\boldsymbol v}(t)=$
{\small{
$$
 e^{(t-t^n)\hat{\boldsymbol B}^{n+\frac{1}{2}}}\left({\boldsymbol v}^n  - \bar{\boldsymbol u} - {\boldsymbol q} + \frac{\nabla \times {\boldsymbol B}^{n+\frac{1}{2}}}{\rho} - (t-t^n)\left(\frac{{\nabla p}^{n+\frac{1}{2}}}{\rho}\right)_\parallel\right) + \bar{\boldsymbol u} + {\boldsymbol q} - \frac{\nabla \times {\boldsymbol B}^{n+\frac{1}{2}}}{\rho}.  
$$}}
\end{remark}
\begin{remark}
In practical simulations, as~\cite{palmroth2018vlasov,vlasiator} the second procedure~\eqref{eq:seconds} and last procedure~\eqref{eq:finals} are equivalent to moving the velocity grids with $$\mp \left( \bar{\boldsymbol u} + {\boldsymbol q} - \frac{\nabla \times {\boldsymbol B}^{n+\frac{1}{2}}}{\rho}\right). $$ And the procedure~\eqref{eq:exact3d} is conducted with the new velocity grids after the second procedure. 
Exact splitting and moving velocity grids have been used in~\cite{palmroth2018vlasov,vlasiator}, but the pressure term is not added in the exact splitting, also 
overcoming the cancellation problem and the conservation of momentum and energy are not investigated. 
\end{remark}

\noindent{\bf Full discretization.}
We denote the above four procedures~\eqref{eq:firsts}-\eqref{eq:finals} with the notations $\mathcal{T}_i, i=0,1,2,3$, respectively, in which the cubic spline based semi-Lagrangian methods~\cite{semi1} are used for the one dimensional translations , i.e., 
$$
{\mathbf f}^{n+1} = \mathcal{T}_3\,\mathcal{T}_2\,\mathcal{T}_1\,\mathcal{T}_0\,{\mathbf f}^{n}.
$$
And we use Fourier spectral method to discretize the magnetic field and pressure.
Specifically, we have the following full discretization.
\begin{equation}
\label{eq:pvbfully}
    \begin{aligned}
        &{\mathbf f}^{n+1} = \mathcal{T}_3\,\mathcal{T}_2\,\mathcal{T}_1\,\mathcal{T}_0\,{\mathbf f}^{n},\\
        &{\mathbf B}_{\boldsymbol i}^{n+1} = {\mathbf B}_{\boldsymbol i}^n + \Delta t\, \nabla \times \mathcal{I}_t \left(\bar{\boldsymbol u}_h \times {\boldsymbol B}_h^{n+\frac{1}{2}} - \frac{\nabla \times {\boldsymbol B}_h^{n+\frac{1}{2}}}{ \rho_h}\times {\boldsymbol B}_h^{n+\frac{1}{2}} + \frac{\nabla p_h}{\rho_h}\right)({\mathbf x}_{\boldsymbol i}),\\
        & \frac{{p}_{\boldsymbol i}^{n+1} - {p}_{\boldsymbol i}^n}{\Delta  t} + \nabla \cdot \mathcal{I}_t({\boldsymbol u}_{e,h}^{n+\frac{1}{2}} { p}_h^{n+\frac{1}{2}})({\mathbf x}_{\boldsymbol i}) + (\gamma - 1)\, {p}_{\boldsymbol i}^{n+\frac{1}{2}} \left( \nabla \cdot \mathcal{I}{\boldsymbol u}_{e,h}^{n+\frac{1}{2}} \right)({\mathbf x}_{\boldsymbol i}) = 0,
    \end{aligned}
\end{equation}
where ${\boldsymbol u}_{e,h} = \bar{\boldsymbol u}_h - \frac{\nabla \times {\boldsymbol B}^{n+\frac{1}{2}}_h}{\rho_h}$, $\mathcal{I}_t{\boldsymbol h}$ gives the triangular interpolation polynomial for each component of the vector valued function ${\boldsymbol h}$ based on its values at grids, and 
the $\bar{\mathbf u}$ is fixed by the following condition, 
\begin{equation}\label{eq:u_int_fully}
     \bar{\mathbf u}_{\boldsymbol i} = \frac{1}{\Delta t} \int_{t^n}^{t^{n+1}}{\mathbf u}_{\boldsymbol i}(t)\, \mathrm{d}{t},
 \end{equation}
where ${\mathbf u}_{\boldsymbol i}(t)$ is the solution of equation 
\begin{equation}\label{eq:bvubarf3_fully}
\frac{\partial {\mathbf u}_{\boldsymbol i}}{\partial t} -\left({\mathbf u}_{\boldsymbol i} - \bar{\mathbf u}_{\boldsymbol i} + \frac{\nabla \times {\boldsymbol B}_h^{n+\frac{1}{2}}({\mathbf x}_{\boldsymbol i})}{\rho_{\boldsymbol i}} \right) \times {\mathbf B}_{\boldsymbol i}^{n+\frac{1}{2}} + \frac{\nabla {p}_h^{n+\frac{1}{2}}({\mathbf x}_{\boldsymbol i})}{\rho_{\boldsymbol i}} = 0, \quad {\mathbf u}_{\boldsymbol i}(t^n) = {\mathbf u}_{\boldsymbol i}^n.
\end{equation}
Similar to the time semi-discretization~\eqref{eq:bvubarf}, we have the explicit formula of the $\bar{\mathbf u}_{\boldsymbol i}$.
\begin{equation}
    \begin{aligned}
\bar{\mathbf u}_{\boldsymbol i} &= {\mathbf u}^n - \frac{\Delta t}{2}\left(\frac{\nabla {p}_h^{n+\frac{1}{2}}({\mathbf x}_{\boldsymbol i})}{\rho_{\boldsymbol i}}\right)_\parallel+ \frac{2\Delta t}{\theta^2}\sin^2\frac{\theta}{2}
\left( \frac{\nabla \times {\boldsymbol B}_h^{n+\frac{1}{2}}({\mathbf x}_{\boldsymbol i})}{\rho_{\boldsymbol i}}  - {\mathbf q}_{\boldsymbol i} \right) \times {\mathbf B}_{\boldsymbol i}^{n+\frac{1}{2}}\\
& - \frac{1}{|{\boldsymbol B}^{n+\frac{1}{2}}|^2}\left(\mathbb{I}_3 - \frac{1}{\theta}  \sin\theta\right) \left( \frac{\nabla \times {\boldsymbol B}_h^{n+\frac{1}{2}}({\mathbf x}_{\boldsymbol i})}{\rho_{\boldsymbol i}}  - {\mathbf q}_{\boldsymbol i}\right)\times {\mathbf B}_{\boldsymbol i}^{n+\frac{1}{2}} \times {\mathbf B}_{\boldsymbol i}^{n+\frac{1}{2}},
 \end{aligned}
\end{equation}
where $\theta = \Delta t |{\mathbf B}_{\boldsymbol i}^{n+\frac{1}{2}}|$.
We denote the numerical solution map of sub-step $pvb$ as $\phi_{pvb}^t$.

\begin{theorem}\label{thm:vv}
    When we assume a periodic boundary condition in space, and an infinite velocity domain,
  the mass, momentum, and energy~\eqref{eq:discretequantities} are conserved by the scheme~\eqref{eq:pvbfully} of sub-step $pvb$, the cancellation problem~\eqref{eq:cancellation} is overcome, and the magnetic field obtained is divergence free.  
\end{theorem}
\begin{proof} 
From~\eqref{eq:pvbfully} we have 
\begin{equation*}
\begin{aligned}
\nabla \cdot {\boldsymbol B}^{n+1}_h & = \nabla \cdot {\boldsymbol B}^{n}_h\\
&- \Delta t\, \nabla \cdot \nabla \times \mathcal{I}_t \left(\bar{\boldsymbol u}_h \times {\boldsymbol B}_h^{n+\frac{1}{2}} - \frac{\nabla \times {\boldsymbol B}_h^{n+\frac{1}{2}}}{\rho_h} \times {\boldsymbol B}_h^{n+\frac{1}{2}}+ \frac{\nabla {p}_h^{n+\frac{1}{2}}}{\rho_h}\right)
\end{aligned}
\end{equation*}
As $\nabla \cdot \nabla \times =0$, we know the magnetic field ${\boldsymbol B}_h$ is divergence free if it holds initially.

The discrete mass, $\mathcal{M}_h^n$, is conserved, as the distribution function is updated by the one dimensional advections in~\eqref{eq:firsts}-\eqref{eq:finals}, each of which conserves the mass due to the first identity in Proposition~\ref{pro:semiidentity}. 

With proposition~\ref{pro:advecmomen}-\ref{pro:semiidentity}, we get the following updated mean velocity correspondingly to the four procedures~\eqref{eq:firsts}-\eqref{eq:finals},
$$
{\mathbf u}^{n,1}_{\boldsymbol i} = {\mathbf u}^n_{\boldsymbol i} - \Delta t \left(\frac{\nabla {p}_h^{n+\frac{1}{2}}({\mathbf x}_{\boldsymbol i})}{\rho_{\boldsymbol i}}\right)_\parallel,  $$ 
$$
{\mathbf u}^{n,2}_{\boldsymbol i} = {\mathbf u}^{n,1}_{\boldsymbol i} - \bar{{\mathbf u}}_{\boldsymbol i} - {\mathbf q}_{\boldsymbol i} + \frac{\nabla \times {\boldsymbol B}_h^{n+\frac{1}{2}}({\mathbf x}_{\boldsymbol i})}{\rho_{\boldsymbol i}} ,  $$ 
$$
{\mathbf u}^{n,3}_{\boldsymbol i} = e^{\Delta t\hat{\mathbf B}^{n+\frac{1}{2}}_{\boldsymbol i}} {\boldsymbol u}^{n,2}_{\boldsymbol i},  $$ 
$$
{\mathbf u}^{n+1}_{\boldsymbol i} =  {\mathbf u}^{n,3}_{\boldsymbol i} + \bar{{\mathbf u}}_{\boldsymbol i} + {\mathbf q}_{\boldsymbol i} - \frac{\nabla \times {\boldsymbol B}_h^{n+\frac{1}{2}}({\mathbf x}_{\boldsymbol i})}{\rho_{\boldsymbol i}} . 
$$
It is easy to verify that ${\mathbf u}_{\boldsymbol i}^{n+1}$ is the solution of the equation~\eqref{eq:bvubarf3_fully} at $t=t^{n+1}$ with the condition ${\mathbf u}_{\boldsymbol i}(t^n) = {\mathbf u}^n_{\boldsymbol i}$. Then by integrating the equation~\eqref{eq:bvubarf3_fully} we have 
$$
{\mathbf u}^{n+1}_{\boldsymbol i} - {\mathbf u}_{\boldsymbol i}^n = \int_{t^n}^{t^{n+1}} \left({\mathbf u}_{\boldsymbol i} - \bar{\mathbf u}_{\boldsymbol i} + \frac{\nabla \times {\boldsymbol B}_h^{n+\frac{1}{2}}({\mathbf x}_{\boldsymbol i})}{\rho_{\boldsymbol i}}  \right) \times {\mathbf B}^{n+\frac{1}{2}}_{\boldsymbol i} - \frac{\nabla {p}_h^{n+\frac{1}{2}}({\mathbf x}_{\boldsymbol i})}{\rho_{\boldsymbol i}}\,  \mathrm{d}t,
$$
where the first two terms cancel to each other due to the condition~\eqref{eq:u_int_fully}. Then we have 
$$
{\mathbf u}^{n+1}_{\boldsymbol i} - {\mathbf u}^n_{\boldsymbol i} = \Delta t\, \frac{\nabla \times {\boldsymbol B}_h^{n+\frac{1}{2}}({\mathbf x}_{\boldsymbol i})}{\rho_{\boldsymbol i}}  \times {\mathbf B}_{\boldsymbol i}^{n+\frac{1}{2}} - \Delta t\, \frac{\nabla {p}_h^{n+\frac{1}{2}}({\mathbf x}_{\boldsymbol i})}{\rho_{\boldsymbol i}},
$$
 i.e., the cancellation problem~\eqref{eq:cancellation} is overcome. Then we can calculate the momentum and get 
 $$
\mathcal{P}_h^{n+1} - \mathcal{P}_h^{n} = \Delta t\, \Delta {\boldsymbol x}  \sum_{\boldsymbol i}  \nabla \times {\boldsymbol B}_h^{n+\frac{1}{2}}({\mathbf x}_{\boldsymbol i})  \times {\mathbf B}_{\boldsymbol i}^{n+\frac{1}{2}} - \Delta t\, \Delta {\boldsymbol x} \sum_{\boldsymbol i}\nabla p_h({\mathbf x}_{\boldsymbol i}),
$$
where the second term is zero, as the Fourier coefficient with zero wave number of $\nabla p_h({\mathbf x})$ is zero. As for the first term, with $\nabla \cdot {\boldsymbol B}_h^{n+\frac{1}{2}}=0$,  we have 
\begin{equation*}
    \begin{aligned}
&\sum_{\boldsymbol i}  \nabla \times {\boldsymbol B}_h^{n+\frac{1}{2}}({\mathbf x}_{\boldsymbol i})  \times {\mathbf B}_{\boldsymbol i}^{n+\frac{1}{2}} = \sum_{\boldsymbol i} \nabla \times {\boldsymbol B}_h^{n+\frac{1}{2}}({\mathbf x}_{\boldsymbol i}) \times {\boldsymbol B}_h^{n+\frac{1}{2}}({\mathbf x}_{\boldsymbol i}) \\
&= \sum_{\boldsymbol i} \nabla \cdot \left( {\boldsymbol B}_h^{n+\frac{1}{2}}{\boldsymbol B}_h^{n+\frac{1}{2}} - \frac{|{\boldsymbol B}_h^{n+\frac{1}{2}}|^2}{2}\mathbb{I}_3\right)({\mathbf x}_{\boldsymbol i}),
 \end{aligned}
\end{equation*}
which is zero due to the orthogonality of the Fourier basis. Then we have proved the momentum conservation.

We define the kinetic energy density as
$$
k_{f,{\boldsymbol i}} = \frac{1}{2} \sum_{{\boldsymbol j}} \left( |v^1_{j_1}|^2 + |v^2_{j_2}|^2 + |v^3_{j_3}|^2 \right)f_{\boldsymbol {ij}}\, \Delta {\boldsymbol v}.
$$
With proposition~\ref{pro:advecmomen}-~\ref{pro:semiidentity}, we get the following updated kinetic energy density correspondingly to the four procedures~\eqref{eq:firsts}-\eqref{eq:finals},
$$
k_{f,{\boldsymbol i}}^{n,1} = k_{f,{\boldsymbol i}}^{n} - \Delta t \, \rho_{\boldsymbol i}\left(\frac{\nabla {p}_h^{n+\frac{1}{2}}({\mathbf x}_{\boldsymbol i})}{\rho_{\boldsymbol i}}\right)_\parallel \cdot \frac{{\mathbf u}^n_{\boldsymbol i} + {\mathbf u}^{n,1}_{\boldsymbol i}}{2}
$$
$$
k_{f,{\boldsymbol i}}^{n,2} = k_{f,{\boldsymbol i}}^{n,1} - \Delta t\, \rho_{\boldsymbol i}\left(\bar{{\mathbf u}}_{\boldsymbol i} + {\mathbf q}_{\boldsymbol i} - \frac{\nabla \times {\boldsymbol B}_h^{n+\frac{1}{2}}({\mathbf x}_{\boldsymbol i})}{\rho_{\boldsymbol i}}  \right) \cdot \frac{{\mathbf u}^{n,1}_{\boldsymbol i} + {\mathbf u}^{n,2}_{\boldsymbol i}}{2} 
$$
$$
k_{f,{\boldsymbol i}}^{n,3} = k_{f,{\boldsymbol i}}^{n,2}
$$
$$
k_{f,{\boldsymbol i}}^{n+1} = k_{f,{\boldsymbol i}}^{n,3} +  \Delta t\, \rho_{\boldsymbol i}\left(\bar{{\mathbf u}}_{\boldsymbol i} + {\mathbf q}_{\boldsymbol i} - \frac{\nabla \times {\boldsymbol B}_h^{n+\frac{1}{2}}({\mathbf x}_{\boldsymbol i})}{\rho_{\boldsymbol i}}  \right) \cdot \frac{{\mathbf u}^{n,3}_{\boldsymbol i} + {\mathbf u}^{n+1}_{\boldsymbol i}}{2}.
$$
It is easy to verify that ${k}_{f,{\boldsymbol i}}^{n+1}$ is the solution of the equation
\begin{equation}
\label{eq:kffully}
    \begin{aligned}
\frac{\partial {k}_{f,{\boldsymbol i}}}{\partial t} &+  \rho_i\bar{\mathbf u}_{\boldsymbol i} \times {\mathbf B}_{\boldsymbol i}^{n+\frac{1}{2}} \cdot  {\mathbf u}_{\boldsymbol i} - {\rho}_{\boldsymbol i}\left(\frac{\nabla \times {\boldsymbol B}_h^{n+\frac{1}{2}}({\mathbf x}_{\boldsymbol i})}{\rho_{\boldsymbol i}}   \times {\mathbf B}_{\boldsymbol i}^{n+\frac{1}{2}} \right) \cdot {\mathbf u}_{\boldsymbol i}\\
& +  {\rho_{\boldsymbol i}}\frac{\nabla {p}_h^{n+\frac{1}{2}}({\mathbf x}_{\boldsymbol i})}{\rho_{\boldsymbol i}}  \cdot {\mathbf u}_{\boldsymbol i} = 0, 
\end{aligned}
\end{equation}
at $t=t^{n+1}$ with the condition ${k}_{f,{\boldsymbol i}}(t^n) = {k}_{f,{\boldsymbol i}}^n$, where the ${\mathbf u}_{\boldsymbol i}$ in~\eqref{eq:kffully} is the solution of the equation~\eqref{eq:bvubarf3_fully} at $t$ with the condition ${\mathbf u}_{\boldsymbol i}(t^n) = {\mathbf u}^n_{\boldsymbol i}$. By integrating the equation~\eqref{eq:kffully} in time we have 
\begin{equation}
\begin{aligned}
k_{f,{\boldsymbol i}}^{n+1} &- k_{f,{\boldsymbol i}}^n = -\int_{t^n}^{t^{n+1}} \rho_{\boldsymbol i}\bar{\mathbf u}_{\boldsymbol i} \times {\mathbf B}_{\boldsymbol i}^{n+\frac{1}{2}} \cdot  {\mathbf u}_{\boldsymbol i}\, \mathrm{d}t\\
&+ \int_{t^n}^{t^{n+1}}  {\rho}_{\boldsymbol i}\left(\frac{\nabla \times {\boldsymbol B}_h^{n+\frac{1}{2}}({\mathbf x}_{\boldsymbol i})}{\rho_{\boldsymbol i}}   \times {\mathbf B}_{\boldsymbol i}^{n+\frac{1}{2}} \right) \cdot {\mathbf u}_{\boldsymbol i} -  {\rho_{\boldsymbol i}}\frac{\nabla {p}_h^{n+\frac{1}{2}}({\mathbf x}_{\boldsymbol i})}{\rho_{\boldsymbol i}}  \cdot {\mathbf u}_{\boldsymbol i}\, \mathrm{d}t,
\end{aligned}
\end{equation}
where the first term is zero due to the condition~\eqref{eq:u_int_fully}. Then with the condition~\eqref{eq:u_int_fully}, we have 
\begin{equation}
\label{eq:smallkf}
\begin{aligned}
k_{f,{\boldsymbol i}}^{n+1} - k_{f,{\boldsymbol i}}^n =
 \Delta t\,  {\rho}_{\boldsymbol i}\left(\frac{\nabla \times {\boldsymbol B}_h^{n+\frac{1}{2}}({\mathbf x}_{\boldsymbol i})}{\rho_{\boldsymbol i}}   \times {\mathbf B}_{\boldsymbol i}^{n+\frac{1}{2}} \right) \cdot \bar{\mathbf u}_{\boldsymbol i} -  \Delta t {\rho_{\boldsymbol i}}\frac{\nabla {p}_h^{n+\frac{1}{2}}({\mathbf x}_{\boldsymbol i})}{\rho_{\boldsymbol i}}  \cdot \bar{\mathbf u}_{\boldsymbol i}.
\end{aligned}
\end{equation}
Summing the equation~\eqref{eq:smallkf} over ${\boldsymbol i}$ and multiplying with $\Delta {\boldsymbol x}$ gives 
\begin{equation}
\label{eq:bigkf}
\begin{aligned}
K_{f,h}^{n+1} - K_{f,h}^n &=
 \underbrace{\Delta t\, \Delta {\boldsymbol x} \sum_{\boldsymbol i} {\rho}_{\boldsymbol i}\left(\frac{\nabla \times {\boldsymbol B}_h^{n+\frac{1}{2}}({\mathbf x}_{\boldsymbol i})}{\rho_{\boldsymbol i}}   \times {\mathbf B}_{\boldsymbol i}^{n+\frac{1}{2}} \right) \cdot \bar{\mathbf u}_{\boldsymbol i}}_{a} \\
 &   \underbrace{-  \Delta t\, \Delta {\boldsymbol x} \sum_{\boldsymbol i} {\rho_{\boldsymbol i}}\frac{\nabla {p}_h^{n+\frac{1}{2}}({\mathbf x}_{\boldsymbol i})}{\rho_{\boldsymbol i}} \cdot \bar{\mathbf u}_{\boldsymbol i}}_{b}.
\end{aligned}
\end{equation}

Multiplying with ${\mathbf B}^{n+\frac{1}{2}}_{\boldsymbol i}$ on both sides of the second equation in~\eqref{eq:pvbfully} and summing with ${\boldsymbol i}$, we have 
\begin{equation*}
\begin{aligned}
K_{B,h}^{n+1} - K_{B,h}^n & = \underbrace{\Delta t\, \Delta {\boldsymbol x} \sum_{\boldsymbol i}\nabla \times {\boldsymbol B}_h^{n+\frac{1}{2}}({\mathbf x}_{\boldsymbol i}) \cdot \left(\bar{\mathbf u}_i \times {\mathbf B}_i^{n+\frac{1}{2}}\right)}_{a}\\
& - \Delta t\, \Delta {\boldsymbol x}  \sum_{\boldsymbol i}\nabla \times {\boldsymbol B}_h^{n+\frac{1}{2}}({\mathbf x}_{\boldsymbol i}) \cdot \left(\frac{\nabla \times {\boldsymbol B}_h^{n+\frac{1}{2}}({\mathbf x}_{\boldsymbol i})}{\rho_{\boldsymbol i}} \times {\mathbf B}_i^{n+\frac{1}{2}} \right)\\
& + \underbrace{\Delta t\, \Delta {\boldsymbol x}  \sum_{\boldsymbol i}\nabla \times {\boldsymbol B}_h^{n+\frac{1}{2}}({\mathbf x}_{\boldsymbol i}) \cdot \frac{\nabla {p}_h^{n+\frac{1}{2}}({\mathbf x}_{\boldsymbol i})}{\rho_{\boldsymbol i}}}_{c},
\end{aligned}
\end{equation*}
where the second term is zero.

Summing the third equation in~\eqref{eq:pvbfully} gives 
\begin{equation*}
\begin{aligned}
K_{p,h}^{n+1}& - K_{p,h}^n =  - \frac{1}{\gamma-1} \Delta t\, \Delta {\boldsymbol x}  \sum_{\boldsymbol i} \left(\nabla \times \mathcal{I}_t({\mathbf u}_{e}^{n+\frac{1}{2}} {\mathbf p}^{n+\frac{1}{2}})\right)({\mathbf x}_{\boldsymbol i})\\
&+ \underbrace{\Delta t\, \Delta {\boldsymbol x} \sum_{\boldsymbol i}\nabla {p}_{h}^{n+\frac{1}{2}}({\mathbf x}_i) \cdot  \bar{\mathbf u}_{i}^{n+\frac{1}{2}}}_{b} \underbrace{-\Delta t\, \Delta {\boldsymbol x}  \sum_{\boldsymbol i} \nabla {p}_{h}^{n+\frac{1}{2}}({\mathbf x}_i) \cdot \frac{\nabla \times {\boldsymbol B}_h^{n+\frac{1}{2}}({\mathbf x}_{\boldsymbol i})}{\rho_{\boldsymbol i}} }_{c}, 
\end{aligned}
\end{equation*}
where the first term is zero.
The remaining terms with the same label cancel with each other, and we have proved the energy conservation.

In conclusion, we have proved the mass, momentum, and energy conservation, the magnetic field is divergence free if it hold initially, also the cancellation problem~\eqref{eq:cancellation} is overcome.
\end{proof}

\begin{remark}
    The scheme~\eqref{eq:pvbfully} is not reversible. The reason is that the cubic spline based semi-Lagrangian method for one dimensional advections is not reversible.
\end{remark}
\begin{comment}
\begin{theorem}
    The scheme~\eqref{eq:pvbfully} is reversible.
\end{theorem}
\begin{proof}
    The reversibility of the scheme of the magnetic field and pressure can be proved similar to Theorem~\ref{thm:reversible1}. For the update of the distribution function in~\eqref{eq:pvbfully}, i.e., 
    $$
    {\mathbf f}^{n+1} = \mathcal{T}_3\mathcal{T}_2\mathcal{T}_1\mathcal{T}_0{\mathbf f}^{n},
    $$
    after exchanging ${\mathbf f}^{n}$ with ${\mathbf f}^{n+1}$ and set $\Delta t = -\Delta t$, we have
    $$
      {\mathbf f}^{n} = \mathcal{T}_3{\mathcal{T}}_2^{-1}\mathcal{T}_1\mathcal{T}_0^{-1}{\mathbf f}^{n+1}, 
    $$
    which is equal to 
    $$
    {\mathbf f}^{n+1} = \mathcal{T}_0\mathcal{T}_1^{-1}\mathcal{T}_2\mathcal{T}_3^{-1} {\mathbf f}^{n}.
    $$
    As $\mathcal{T}_3^{-1} = \mathcal{T}_1, \mathcal{T}_1^{-1} = \mathcal{T}_3$, and with the Remark~\ref{rm:commutable} we know $\mathcal{T}_0$ is commutative with the shifted rotation $\mathcal{T}_1^{-1}\mathcal{T}_2\mathcal{T}_3^{-1}$ , we have 
$$
    {\mathbf f}^{n+1} = \mathcal{T}_3\mathcal{T}_2\mathcal{T}_1\mathcal{T}_0{\mathbf f}^{n},
    $$
    and finished the proof.
\end{proof}
\end{comment}

\begin{remark}
    In the case of isothermal/adiabatic electrons, there is no pressure equation in the hybrid model, instead the pressure $p$ is determined by the density $\rho$ directly, i.e., $p = \kappa \rho^\gamma$. Similar to theorem~\ref{thm:vv}, the scheme~\eqref{eq:pvbfully} without the scheme of the pressure but with the relation $p = \kappa \rho^\gamma$ can be proved to preserve the mass and momentum. The energy, including  a non-quadratic thermal energy term, is not conserved~\cite{yingzhe1}. Also the magnetic field is divergence free if it is initially, and the cancellation problem~\eqref{eq:cancellation} is overcome.
\end{remark}

\subsection{Sub-step $xv$}
The equation of this sub-step is
$$
\frac{\partial f}{\partial t} = - \boldsymbol v \cdot \nabla_x f 
$$
for which only one dimensional translations need to be done, for which 
we can choose to use the 
Fourier spectral method~\eqref{eq:velospec}, or the cubic spline based semi-Lagrangian method~\eqref{eq:cubicsemi} in each spatial direction. 
We denote the numerical solution map of sub-step $pxv$ as $\phi_{xv}^t$. 
\begin{theorem}\label{thm:xv} When we assume a periodic boundary condition in space,
    the scheme~\eqref{eq:velospec} or the scheme~\eqref{eq:cubicsemi} of sub-step $xv$ conserves mass, momentum~\eqref{eq:discretequantities}, and energy~\eqref{eq:discretequantities} exactly. 
\end{theorem}
\begin{proof}
By the Lemma~\ref{lemmapu} when the scheme~\eqref{eq:velospec} is used, and by Proposition~\ref{pro:semiidentity} when the cubic spline based semi-Lagrangian scheme~\eqref{eq:cubicsemi} is used, we get mass, momentum, and energy conservation. 
\end{proof}
By the composition methods~\cite{HLW}, we have the following first order Lie splitting and second order Strang splitting methods~\cite{Strang},
$$
\phi^{\Delta t}_{xv}\,\phi^{\Delta t}_{pvb}, \quad \phi^{\frac{{\Delta t}}{2}}_{xv}\,\phi^{\Delta t}_{pvb}\,\phi^{\frac{{\Delta t}}{2}}_{xv}.
$$

\begin{remark}
    Let us remark that other numerical methods can also be adopted to derive the same conservation properties. For example, the central finite difference method~\cite{chacon1} can be used for the discretization of the magnetic field and pressure. 
\end{remark}

\section{1D-2V reduced model}\label{sec:reduced}
Here, we consider a reduced model which is one dimensional in space and two dimensional in velocity, in which the distribution function is $f(t, x_1, v_1, v_2)$, and the magnetic field is in the form of ${\boldsymbol B} = (0,0,B_3(x_1))^\top$. The reduced model reads
\begin{equation}
\label{eq:reduced}
\begin{aligned}
\frac{\partial f}{\partial t} &= -v_1 \frac{\partial f}{\partial x_1} - v_2 B_3 \frac{\partial f}{\partial v_1} + v_1 B_3 \frac{\partial f}{\partial v_2} + u_2 B_3 \frac{\partial f}{\partial v_1} - u_1 B_3 \frac{\partial f}{\partial v_2} \\
& +  \frac{1}{\rho}\frac{\partial p}{\partial x_1} \frac{\partial f}{\partial v_1} + \frac{1}{\rho} B_3 \frac{\partial B_3}{\partial x_1} \frac{\partial f}{\partial v_1}\\
\frac{\partial B_3}{\partial t} &= - \frac{\partial}{\partial x_1}( u_1 B_3),\quad \rho = \int f\, \mathrm{d}{\boldsymbol v}, \quad u_1 = \frac{1}{\rho}\int v_1 f\, \mathrm{d}{\boldsymbol v}, \quad u_2 = \frac{1}{\rho}\int v_2 f\, \mathrm{d}{\boldsymbol v},\\
\frac{\partial p}{\partial t} &+ \frac{\partial}{\partial x_1} \left(\frac{\int v_1 f\, \mathrm{d}{\boldsymbol v}}{\rho} p\right) + (\gamma - 1) \,p\, \frac{\partial}{\partial x_1} \frac{\int v_1 f\, \mathrm{d}{\boldsymbol v}}{\rho} = 0.
\end{aligned}
\end{equation}
For this reduced model~\eqref{eq:reduced}, we have the following two sub-steps after Poisson splitting~\cite{2020en,yingzhe1}. As the three dimensional case, we use uniform grids in phase-space with grid numbers $M_1, N_1, N_2$ in directions $x_1, v_1, v_2$, respectively.

\noindent{\bf Sub-step $pvb$.}
The sub-step $pvb$ is
\begin{equation*}
\begin{aligned}
\frac{\partial f}{\partial t}&=  - v_2 B_3 \frac{\partial f}{\partial v_1} + v_1 B_3 \frac{\partial f}{\partial v_2} + u_2 B_3 \frac{\partial f}{\partial v_1} - u_1 B_3 \frac{\partial f}{\partial v_2} \\
& +  \frac{1}{\rho}\frac{\partial p}{\partial x_1} \frac{\partial f}{\partial v_1} + \frac{1}{\rho} B_3 \frac{\partial B_3}{\partial x_1} \frac{\partial f}{\partial v_1}\\
\frac{\partial B_3}{\partial t} &= - \frac{\partial}{\partial x_1}( u_1 B_3),\quad \rho = \int f\, \mathrm{d}{\boldsymbol v}, \quad u_1 = \frac{1}{\rho}\int v_1 f\, \mathrm{d}{\boldsymbol v}, \quad u_2 = \frac{1}{\rho}\int v_2 f\, \mathrm{d}{\boldsymbol v},\\
\frac{\partial p}{\partial t} &+ \frac{\partial}{\partial x_1} \left(\frac{\int v_1 f\, \mathrm{d}{\boldsymbol v}}{\rho} p\right) + (\gamma - 1) p \frac{\partial}{\partial x_1} \frac{\int v_1 f\, \mathrm{d}{\boldsymbol v}}{\rho} = 0.
\end{aligned}
\end{equation*}
The similar scheme to~\eqref{eq:pvbfully} can be constructed, in which the two dimensional rotation is solved by the exact splitting presented in~\cite{exact3}. The scheme overcomes the cancellation problem, conserves mass, momentum, and energy.

\noindent{\bf Sub-step $xv$.} The sub-step $xv$ is
\begin{equation}
\label{eq:1d2vpxv}
\begin{aligned}
&\frac{\partial f}{\partial t} = -v_1 \frac{\partial f}{\partial x_1}.
\end{aligned}
\end{equation}
This sub-step is solved with Fourier spectral method~\eqref{eq:velospec} or the scheme~\eqref{eq:cubicsemi}, and the mass, momentum, and energy are conserved.

\section{Numerical experiments}\label{sec:Numerical}
In this section, firstly we check the time accuracy order of the schemes constructed. Then we use the 1D-2V reduced model~\eqref{eq:reduced} to simulate Landau damping and Bernstein waves with the schemes constructed. The Fourier spectral method~\eqref{eq:velospec} is adopted for solving sub-step~\eqref{eq:1d2vpxv}. The tolerance of the Picard iteration is $10^{-14}$.

%\subsection{Performance of the linear solvers of sub-step $bv$}
%Here we check the 1D and 3D linear solvers of sub-step $bv$.

\subsection{Time accuracy order and reversibility}
Here we firstly check the time accuracy order by 1D-2V dimensional simulations. The initial conditions are:
\begin{equation*}
\begin{aligned}
   f =  \frac{1}{\pi v_T^2} (1 + \delta \rho) e^{-\frac{|v_1-0.1|^2}{v_T^2} - \frac{|v_2-0.2|^2}{v_T^2}   }, \, B_3 = 1 + 0.01\sin(2x_1), \, p = 0.09(1+\delta \rho)^{5/3},
\end{aligned}
\end{equation*}
where $\delta \rho = 0.01\sin(2x_1)$, and  $v_T = 0.4$. 
The computational parameters are as follows, 
grid number $16 \times 128 \times 128$, and domain  $[0, \pi] \times [-2.5,2.5] \times [-2.5, 2.5]$.
We use the Strang splitting $\phi^{\frac{{\Delta t}}{2}}_{xv}\phi^{\Delta t}_{pvb}\phi^{\frac{{\Delta t}}{2}}_{xv}$. In table~\ref{tab:error_convergence}, we present the $l_1$ error of $f$, $B_3$, and $p$ at $t=0.1$, the reference solution is obtained by a very small time step $0.001$. We can see when $\Delta t$ becomes smaller, the order of the error convergence becomes closer to 2, which indicates that the Strang splitting gives a second order accuracy in time.
\begin{table}[h]
    \centering
    \setlength{\tabcolsep}{2pt} % 调整列间距
    \begin{tabular}{ccccccc}
        \toprule
        $\Delta t$ & Error of $f$ & Order of $f$ & Error of $B_3$ & Order of $B_3$ & Error of $p$ & Order of $p$ \\
        \midrule
        $0.1$   & $16.97$ & -- & $1.728 \times 10^{-2}$ & -- & $2.617 \times 10^{-3}$ & -- \\
        $0.05$   & $4.734$ & $1.84$ & $5.223 \times 10^{-3}$ & $1.73$ & $7.906 \times 10^{-4}$& $1.73$\\
        $0.025$  & $1.232$ & $1.94$ & $1.379 \times 10^{-3}$ & $1.92$ & $2.086 \times 10^{-4}$ & $1.92$\\
        $0.0125$ & $0.3133$ & $1.98$  & $3.48 \times 10^{-4}$ & $1.99$ & $5.267 \times 10^{-5}$ & $1.99$\\
        $0.00625$ & $0.0805$ & $1.96$ & $8.564 \times 10^{-5}$ & $2.02$ & $1.296 \times 10^{-5}$ & $2.02$\\
        \bottomrule
    \end{tabular}
    \caption{The $l_1$ errors and convergence orders of the unknowns for different time step sizes.}
    \label{tab:error_convergence}
\end{table}

To validate the reversibility of the time semi-discretization~\eqref{eq:bvubarf}, we run the Strang splitting $\phi^{\frac{{\Delta t}}{2}}_{xv}\phi^{\Delta t}_{pvb}\phi^{\frac{{\Delta t}}{2}}_{xv}$ with grids $17 \times 64 \times 64$, in which all the one dimensional advections for the distribution functions are done with Fourier spectral methods (instead of the cubic spline interpolations) to reduce the influence from the error of phase-space discretization. We run with the above initial condition and time step size $\Delta t$ for 20 time steps, and then run with the numerical solution obtained as initial condition and the time step size $-\Delta t$ for another 20 steps, the error between the initial condition and final results is presented in Table.~\ref{tab:reversibility}, we can see the error is very small and the reversibility of the time semi-discretization~\eqref{eq:bvubarf} is validated. 

\begin{table}[h]
    \centering
    \setlength{\tabcolsep}{6pt} % 调整列间距
    \begin{tabular}{cccc}
        \toprule
        $\Delta t$ & Error of $f$  & Error of $B_3$  & Error of $p$  \\
        \midrule
        $0.1$   & $2.84 \times 10^{-9}$  & $3.07 \times 10^{-12}$ &  $4.65 \times 10^{-13}$  \\
        $0.05$ & $3.67\times 10^{-11}$ & $8.10 \times 10^{-15}$ &  $1.48 \times 10^{-15}$ \\
        $0.02$ & $3.38\times 10^{-11}$ & $5.55 \times 10^{-15}$ &  $8.88 \times 10^{-16}$ \\
        \bottomrule
    \end{tabular}
    \caption{The $l_1$ errors between the initial condition and numerical results of the unknowns for different time step sizes.}
    \label{tab:reversibility}
\end{table}

\subsection{Landau damping}
Next we simulate ion Landau damping by 1D-2V dimensional simulations without background magnetic field. Initial conditions are:
\begin{equation*}
\begin{aligned}
   f(x_1,v_1,v_2) =  \frac{1}{\pi v_T^2} (1 + \delta \rho) e^{-\frac{|v_1|^2}{v_T^2} - \frac{|v_2|^2}{v_T^2}   }, \quad B_3 = 0,
\end{aligned}
\end{equation*}
where $\delta \rho = 0.01\sin(0.4x_1)$. 
Grid numbers are $32 \times 128 \times 64$, the simulation domain is  $[0, 5\pi] \times [-8,8] \times [-8, 8]$, the time step size is $\Delta t = 0.1$, and $v_T = 1.4142$. 
As there is not magnetic field, 
we regard the ${\nabla p}/{\rho}$ parallel to the magnetic field, and the procedures~\eqref{eq:seconds}-\eqref{eq:finals} are not needed in $\phi^{t}_{pvb}$.
 
See the numerical results with the Strang splitting
$\phi^{\frac{\Delta t}{2}}_{xv} \phi^{\Delta t}_{pvb}\phi^{\frac{\Delta t}{2}}_{xv} $~\cite{HLW} for the case of the isothermal electrons with $\kappa = 6.25$ in Fig.~\ref{fig:landau}. As the pressure equation is not needed to solve in this case, the sub-step $pvb$ is the usual semi-Lagrangian advection in velocity with the electron pressure gradient and no Picard iteration is needed. $\|\rho-1\|_{2}^2$ with $\rho = \int f \,\mathrm{d}{\boldsymbol v}$ decays exponentially with a rate close to the theoretical one, which is obtained from the dispersion relation formula in~\cite{Pegasus, chacon1}. As semi-Lagrangian methods do not have noise, compared to the results of particle-in-cell method~\cite{yingzhe1}, a clear decay rate is observed.
The energy error is around $10^{-6}$. The mass error is around $10^{-14}$, and
momentum errors in $v_1$ and $v_2$ directions at the level of $10^{-14}, 10^{-17}$, respectively.

\begin{figure}[htbp]
\center{
\includegraphics[scale=0.25]{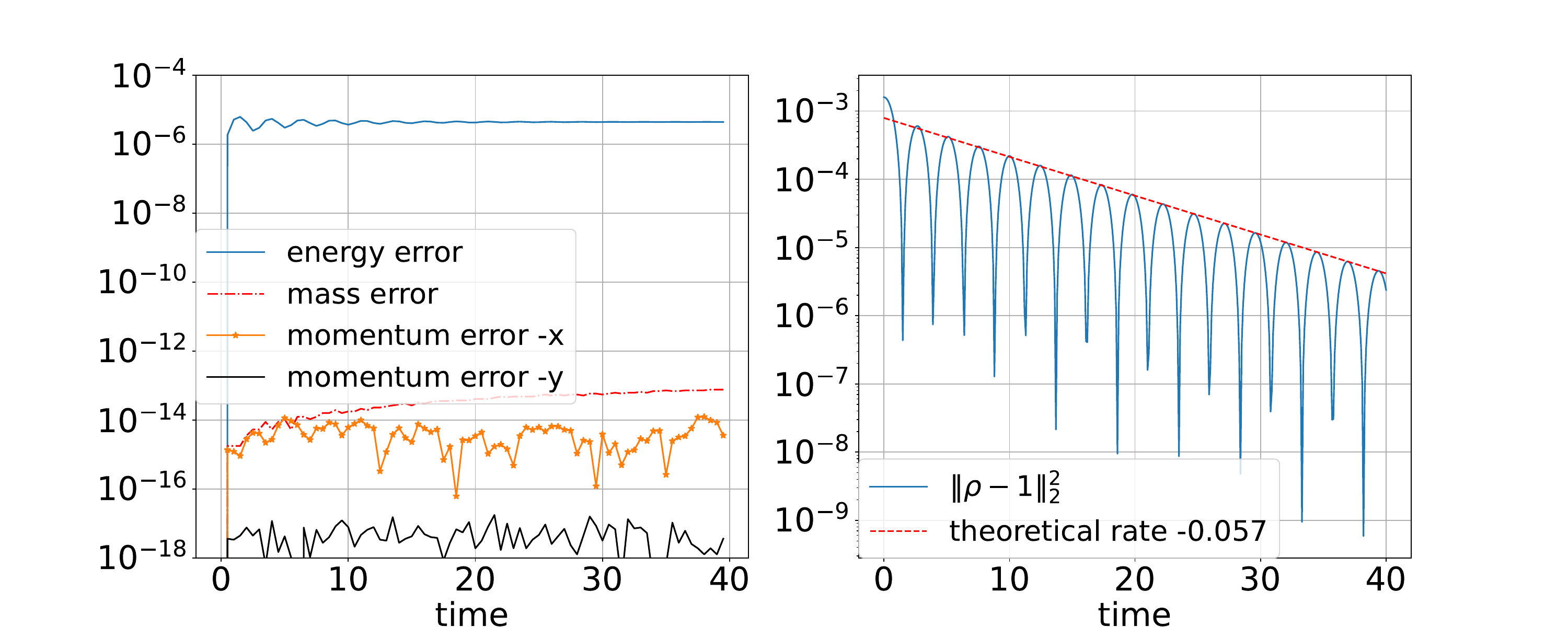}
}
\caption{{\bf Landau damping} The formulation with the isothermal electrons. Time evolutions of energy error, mass error, momentum errors, and $\|\rho - 1\|_2^2$.}
\label{fig:landau}
\end{figure}

Next we consider the case with the electron pressure equation and $\gamma = 5/3$. We use the Strang splitting
$\phi^{\frac{\Delta t}{2}}_{xv} \phi^{\Delta t}_{pvb}\phi^{\frac{\Delta t}{2}}_{xv} $~\cite{HLW}. The initial distribution function is the same as the isothermal electron case, and the initial electron pressure is 
$
p(x_1) = 6.25/\gamma (1+\delta \rho)^\gamma
$. This case is equivalent to the formulation with adiabatic electrons $p =  6.25/\gamma (1+\delta \rho)^\gamma$ and without explicitly using the pressure equation. The $\gamma$ in the denominator of $p =  6.25/\gamma (1+\delta \rho)^\gamma$  ensures that the linearized model coincides with the linearized isothermal case considered above, leading to the same decay rate for the perturbation of $\rho$. In the simulations, around 15 picard iterations are used in the sub-step $pvb$.
 See the numerical results in Fig.~\ref{fig:landau_pressure}. We can see the time evolution of $\|\rho - 1\|_{2}^2$ decays with the same rate as the isothermal electron case, mass error is around $10^{-14}$, and momentum errors are $10^{-14}, 10^{-17}$ in $x, y$ directions, respectively. The energy is also conserved very well, with an error around $10^{-12}$.

\begin{figure}[htbp]
\center{
\includegraphics[scale=0.25]{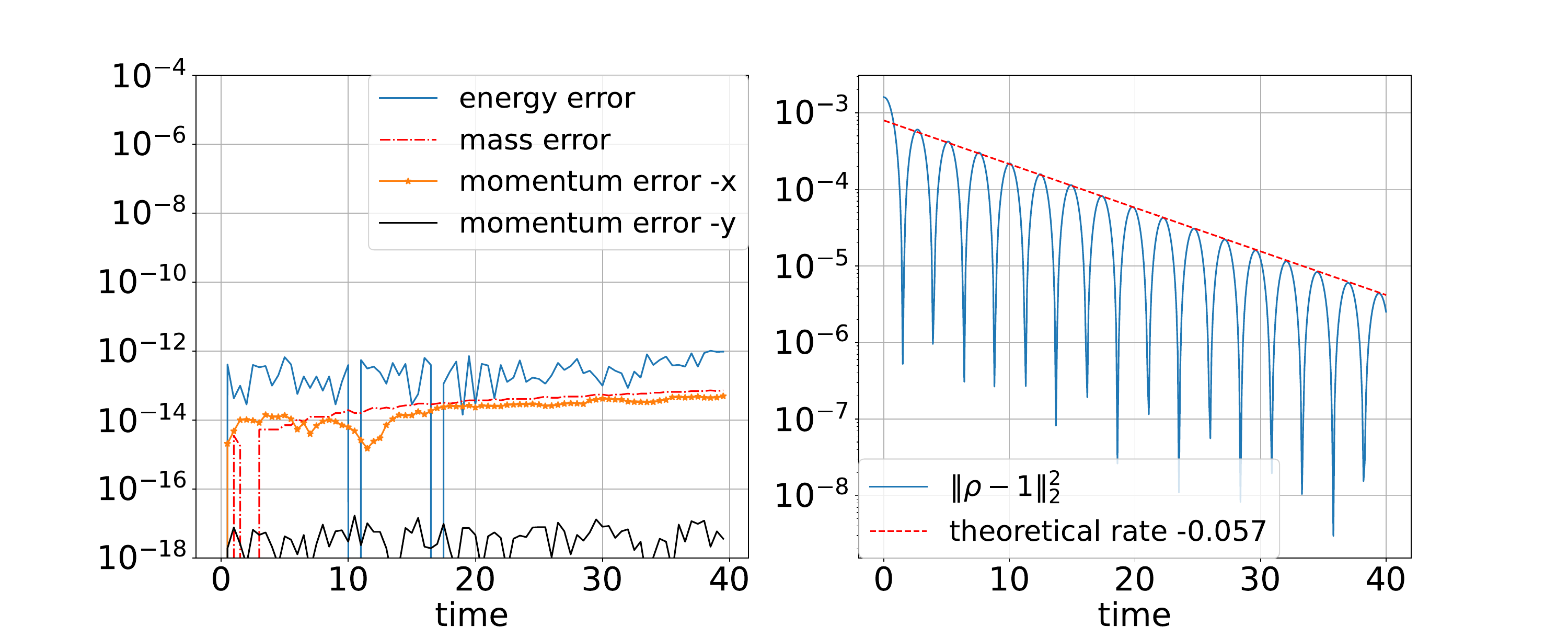}
}
\caption{{\bf Landau damping} The formulation with the electron pressure equation. Time evolutions of energy error, mass error, momentum errors, and $\|\rho - 1\|_2^2$.}
\label{fig:landau_pressure}
\end{figure}

\subsection{Bernstein waves}
We use the 1D-2V reduced model for simulating the Bernstein waves. The initial conditions of the isothermal electron case are:
\begin{equation*}
\begin{aligned}
 f(x_1,v_1,v_2) =  \frac{1}{\pi v_T^2} e^{-\frac{|v_1|^2}{v_T^2} - \frac{|v_2|^2}{v_T^2}}, \quad B_3(x_1) = 1 + 10^{-5}\sum_{k=1}^{32} \sin\left(\frac{k}{2}x_1\right).
\end{aligned}
\end{equation*}
The computational parameters are: 
grid number $64 \times 128 \times 128$, domain $[0, 4\pi]\times [-3, 3] \times [-3, 3]$, time step size $\Delta t = 0.05$, $v_T = 0.4$, final computation time $80$. 

For the isothermal electron case with $\kappa = 0.09$, 
we use the Strang splitting
$\phi^{\frac{\Delta t}{2}}_{xv} \phi^{\Delta t}_{pvb}\phi^{\frac{\Delta t}{2}}_{xv} $~\cite{HLW}.
We show the the time evolutions of energy error, momentum errors, and mass error, and the dispersion relations of Bernstein waves in Fig.~\ref{fig:bernstein_no_pressure}. The errors of mass, momentum, and energy are around $10^{-13}$, $10^{-13}$, and $10^{-11}$, respectively. 
The red lines in the second figure of Fig.~\ref{fig:bernstein_no_pressure} denote the analytical dispersion relation obtained by the Python package HYDROS~\cite{disp}, and we can see that our numerical results fit in well with the analytical dispersion relation.

Next, we consider the case with the electron pressure equation and $\gamma = 5/3$ with the Strang splitting
$\phi^{\frac{\Delta t}{2}}_{xv} \phi^{\Delta t}_{pvb}\phi^{\frac{\Delta t}{2}}_{xv} $~\cite{HLW}. 
We use the same initial condition as above and the initial value of the electron pressure is $p(x_1) = \kappa$. Note that with the initial condition we choose, the hybrid model with the electron pressure equation is equivalent to the hybrid model with adiabatic electrons, $\gamma = 5/3$, and $p=\kappa \rho^\gamma$. On the numerical side, we should also preserve the relation $p=\kappa \rho^\gamma$ as much as possible.

The computation parameters are the same as isothermal electron case. The results are presented in Fig.~\ref{fig:bernstein_with_pressure}. We can see the energy is conserved with an error around $10^{-12}$. Mass and momentum errors are around $10^{-13}$.  In the second figure, we can see several branches of the numerical dispersion relation, each branch becomes flat when the wave number becomes large, which is classical in the dispersion relation of Bernstein waves. These branches fit in well with the red lines, which denote the analytical dispersion relation of the hybrid model with adiabatic electrons by HYDROS~\cite{disp}.  The error in $l_1$ norm (around $10^{-8}$) of the relation $p = \kappa \rho^\gamma$ is shown in the first figure in Fig.~\ref{fig:bernstein_with_pressure}. 
It is also important to note that for general electrons, i.e., non-isothermal and non-adiabatic electrons, there is no such relation that must be conserved numerically.

\begin{figure}[htbp]
\center{
\includegraphics[scale=0.25]{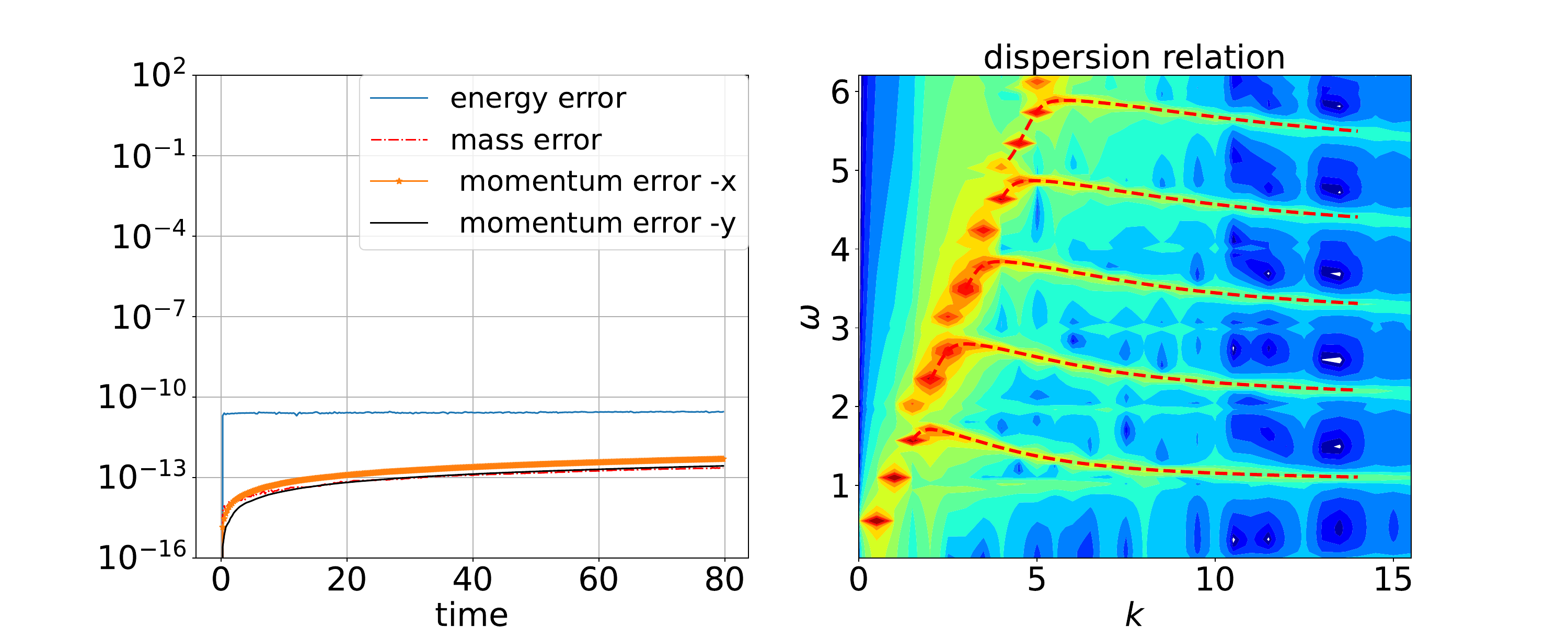}
}
\caption{{\bf Bernstein waves} The formulation with isothermal electrons. Left figure: time evolution of the errors of mass. momentum, and energy. Right figure: the dispersion relation of the Bernstein waves. }
\label{fig:bernstein_no_pressure}
\end{figure}

\begin{figure}[htbp]
\center{
\includegraphics[scale=0.25]{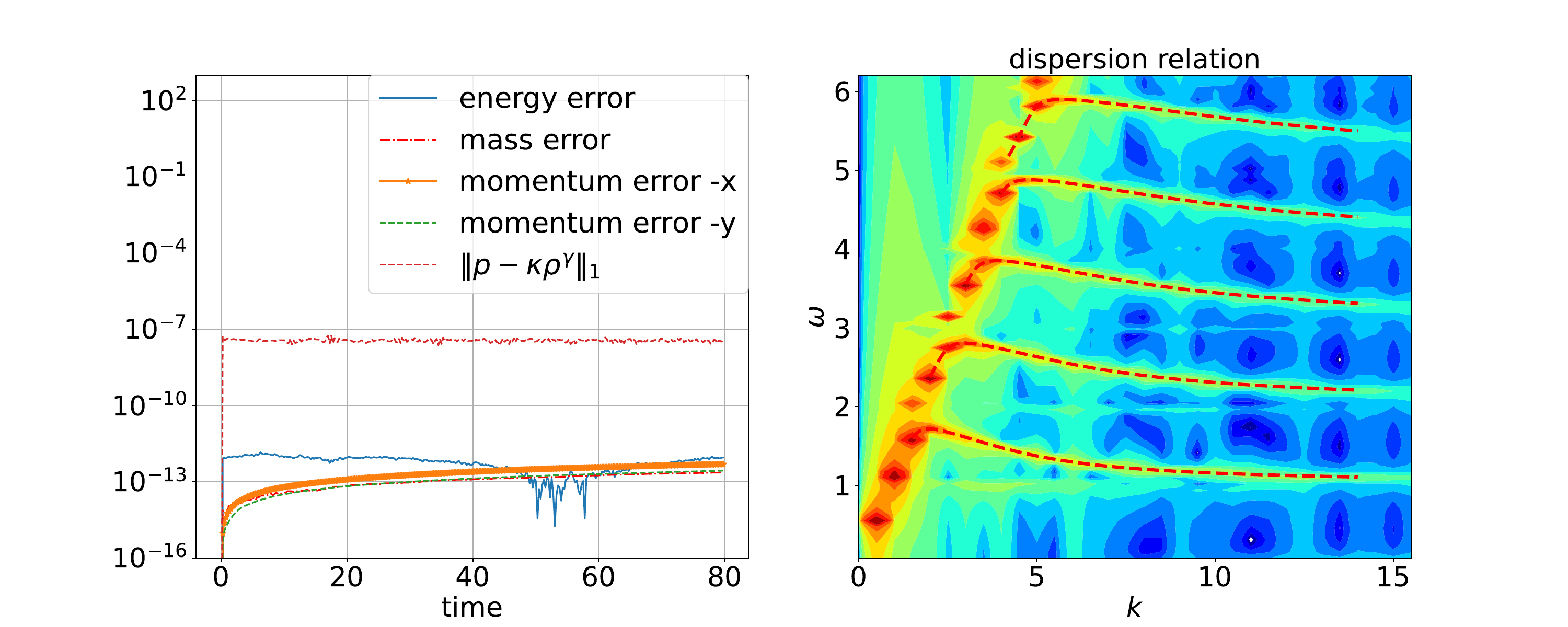}
}
\caption{{\bf Bernstein waves} The formulation with the electron pressure equation. Left figure: Time evolution of $\|p - \kappa \rho^\gamma \|_{1}$ and the errors of mass, momentum, and energy.  Right figure: the dispersion relation of the Bernstein waves. }
\label{fig:bernstein_with_pressure}
\end{figure}

\section{Conclusion}\label{sec:conclusion}
In this work, the cubic spline based semi-Lagrangian methods~\cite{semi1} are used to solve the hybrid model with kinetic ions and massless electrons. Thanks to the exact splitting constructed,
the Poisson splitting~\cite{2020en}, based on the Poisson bracket~\cite{Poisson, Tronci}, yields only two sub-steps. For the complicated sub-step involving the update of the kinetic velocity with the mean fluid velocity,  a specially designed mean velocity is chosen to ensure good conservation properties. The entire scheme is efficient, as it requires only one-dimensional advections for the distribution functions, with nonlinear iterations applied solely to the field unknowns. The scheme conserves mass, momentum, and energy, while also preserving the divergence-free condition of the magnetic field and overcome the cancellation problem~\cite{cancellation}. These conservation properties are validated through numerical experiments on Landau damping and Bernstein waves. 

At present, we are implementing the algorithms proposed in this work within Gysela-X~\cite{grandgirard20165d}, with the aim of subsequent physical applications.
Other future work includes applying the exact splitting to other electro-magnetic models,
developing efficient solvers for the nonlinear systems, performing simulations on domains with complex geometries, exploring other discretization methods, such as the numerical methods in the framework of finite element exterior calculus~\cite{feec} and discrete exterior calculus~\cite{dec}.

\section*{Acknowledgments}
Simulations in this work were performed on Max Planck Computing \& Data Facility (MPCDF). The authors would like to acknowledge Omar Maj for the kind discussions about Bernstein waves and Stefan Possanner for the fruitful discussions on this model over the years.

\bibliographystyle{siamplain}
\bibliography{references}

\begin{thebibliography}{10}

\bibitem{feec}
{\sc D.~N. Arnold, R.~S. Falk, and R.~Winther}, {\em Finite element exterior calculus, homological techniques, and applications}, Acta numerica, 15 (2006), pp.~1--155.

\bibitem{exact2}
{\sc J.~Bernier}, {\em Exact splitting methods for semigroups generated by inhomogeneous quadratic differential operators}, Foundations of Computational Mathematics, 21 (2021), pp.~1401--1439.

\bibitem{exact3}
{\sc J.~Bernier, F.~Casas, and N.~Crouseilles}, {\em Splitting methods for rotations: application to {V}lasov equations}, SIAM Journal on Scientific Computing, 42 (2020), pp.~A666--A697.

\bibitem{exact1}
{\sc J.~Bernier, N.~Crouseilles, and Y.~Li}, {\em Exact splitting methods for kinetic and {S}chr{\"o}dinger equations}, Journal of Scientific Computing, 86 (2021), pp.~1--35.

\bibitem{cai2021hybrid}
{\sc H.-b. Cai, X.-x. Yan, P.-l. Yao, and S.-p. Zhu}, {\em Hybrid fluid--particle modeling of shock-driven hydrodynamic instabilities in a plasma}, Matter and Radiation at Extremes, 6 (2021).

\bibitem{martin}
{\sc M.~Campos~Pinto, K.~Kormann, and E.~Sonnendr{\"u}cker}, {\em Variational framework for structure-preserving electromagnetic particle-in-cell methods}, Journal of Scientific Computing, 91 (2022), p.~46.

\bibitem{chenbao}
{\sc B.~Chen and A.~Kaufman}, {\em 3d volume rotation using shear transformations}, Graphical Models, 62 (2000), pp.~308--322.

\bibitem{semi2}
{\sc C.-Z. Cheng and G.~Knorr}, {\em The integration of the {V}lasov equation in configuration space}, Journal of Computational Physics, 22 (1976), pp.~330--351.

\bibitem{Nicolas}
{\sc N.~Crouseilles, L.~Einkemmer, and E.~Faou}, {\em Hamiltonian splitting for the {V}lasov--{M}axwell equations}, Journal of Computational Physics, 283 (2015), pp.~224--240.

\bibitem{crouseilles2010conservative}
{\sc N.~Crouseilles, M.~Mehrenberger, and E.~Sonnendr{\"u}cker}, {\em Conservative semi-{L}agrangian schemes for {V}lasov equations}, Journal of Computational Physics, 229 (2010), pp.~1927--1953.

\bibitem{dec}
{\sc M.~Desbrun, A.~N. Hirani, M.~Leok, and J.~E. Marsden}, {\em Discrete exterior calculus}, arXiv preprint math/0508341,  (2005).

\bibitem{filbet}
{\sc F.~Filbet, E.~Sonnendr{\"u}cker, and P.~Bertrand}, {\em Conservative numerical schemes for the {V}lasov equation}, Journal of Computational Physics, 172 (2001), pp.~166--187.

\bibitem{CAMELIA}
{\sc L.~Franci, P.~Hellinger, M.~Guarrasi, C.~Chen, E.~Papini, A.~Verdini, L.~Matteini, and S.~Landi}, {\em Three-dimensional simulations of solar wind turbulence with the hybrid code {C}{A}{M}{E}{L}{I}{A}}, in Journal of Physics: Conference Series, vol.~1031, IOP Publishing, 2018, p.~012002.

\bibitem{grandgirard20165d}
{\sc V.~Grandgirard, J.~Abiteboul, J.~Bigot, T.~Cartier-Michaud, N.~Crouseilles, G.~Dif-Pradalier, C.~Ehrlacher, D.~Esteve, X.~Garbet, P.~Ghendrih, et~al.}, {\em A 5d gyrokinetic full-f global semi-lagrangian code for flux-driven ion turbulence simulations}, Computer physics communications, 207 (2016), pp.~35--68.

\bibitem{HLW}
{\sc E.~Hairer, M.~Hochbruck, A.~Iserles, and C.~Lubich}, {\em Geometric numerical integration}, Oberwolfach Reports, 3 (2006), pp.~805--882.

\bibitem{he}
{\sc Y.~He, H.~Qin, Y.~Sun, J.~Xiao, R.~Zhang, and J.~Liu}, {\em Hamiltonian time integrators for {V}lasov-{M}axwell equations}, Physics of Plasmas, 22 (2015).

\bibitem{he2015volume}
{\sc Y.~He, Y.~Sun, J.~Liu, and H.~Qin}, {\em Volume-preserving algorithms for charged particle dynamics}, Journal of Computational Physics, 281 (2015), pp.~135--147.

\bibitem{heyang}
{\sc Y.~He, Y.~Sun, H.~Qin, and J.~Liu}, {\em Hamiltonian particle-in-cell methods for {V}lasov-{M}axwell equations}, Physics of Plasmas, 23 (2016).

\bibitem{directional}
{\sc F.~Huot, A.~Ghizzo, P.~Bertrand, E.~Sonnendr{\"u}cker, and O.~Coulaud}, {\em Instability of the time splitting scheme for the one-dimensional and relativistic {V}lasov--{M}axwell system}, Journal of Computational Physics, 185 (2003), pp.~512--531.

\bibitem{2020en}
{\sc K.~Kormann and E.~Sonnendr{\"u}cker}, {\em Energy-conserving time propagation for a structure-preserving particle-in-cell {V}lasov--{M}axwell solver}, Journal of Computational Physics, 425 (2021), p.~109890.

\bibitem{dual}
{\sc K.~Kormann and E.~Sonnendr{\"u}cker}, {\em A {D}ual {G}rid {G}eometric {E}lectromagnetic {P}article in {C}ell {M}ethod}, SIAM Journal on Scientific Computing, 46 (2024), pp.~B621--B646.

\bibitem{GEMPIC}
{\sc M.~Kraus, K.~Kormann, P.~J. Morrison, and E.~Sonnendr{\"u}cker}, {\em G{E}{M}{P}{I}{C}: geometric electromagnetic particle-in-cell methods}, Journal of Plasma Physics, 83 (2017), p.~905830401.

\bibitem{Pegasus}
{\sc M.~W. Kunz, J.~M. Stone, and X.-N. Bai}, {\em Pegasus: a new hybrid-kinetic particle-in-cell code for astrophysical plasma dynamics}, Journal of Computational Physics, 259 (2014), pp.~154--174.

\bibitem{le}
{\sc A.~Le, A.~Stanier, L.~Yin, B.~Wetherton, B.~Keenan, and B.~Albright}, {\em Hybrid-{V}{P}{I}{C}: {A}n open-source kinetic/fluid hybrid particle-in-cell code}, Physics of Plasmas, 30 (2023).

\bibitem{yingzhe1}
{\sc Y.~Li, M.~Campos~Pinto, F.~Holderied, S.~Possanner, and E.~Sonnendr{\"u}cker}, {\em Geometric particle-in-cell discretizations of a plasma hybrid model with kinetic ions and mass-less fluid electrons}, Journal of Computational Physics, 498 (2024), p.~112671.

\bibitem{yingzhe6d}
{\sc Y.~Li, Y.~He, Y.~Sun, J.~Niesen, H.~Qin, and J.~Liu}, {\em Solving the {V}lasov--{M}axwell equations using {H}amiltonian splitting}, Journal of Computational Physics, 396 (2019), pp.~381--399.

\bibitem{yingzhe2}
{\sc Y.~Li, F.~Holderied, S.~Possanner, and E.~Sonnendr{\"u}cker}, {\em Canonical variables based numerical schemes for hybrid plasma models with kinetic ions and massless electrons}, Journal of Computational Physics, 505 (2024), p.~112916.

\bibitem{Poisson}
{\sc Y.~Li, P.~J. Morrison, S.~Possanner, and E.~Sonnendr{\"u}cker}, {\em On the {P}oisson brackets of the hybrid model with kinetic ions and massless electrons}.

\bibitem{hongtao}
{\sc H.~Liu, X.~Cai, Y.~Cao, and G.~Lapenta}, {\em An efficient energy conserving semi-lagrangian kinetic scheme for the {V}lasov-{A}mp{\`e}re system}, Journal of Computational Physics, 492 (2023), p.~112412.

\bibitem{current2D}
{\sc A.~P. Matthews}, {\em Current advance method and cyclic leapfrog for 2{D} multispecies hybrid plasma simulations}, Journal of Computational Physics, 112 (1994), pp.~102--116.

\bibitem{splitting}
{\sc R.~I. McLachlan and G.~R.~W. Quispel}, {\em Splitting methods}, Acta Numerica, 11 (2002), pp.~341--434.

\bibitem{palmroth2018vlasov}
{\sc M.~Palmroth, U.~Ganse, Y.~Pfau-Kempf, M.~Battarbee, L.~Turc, T.~Brito, M.~Grandin, S.~Hoilijoki, A.~Sandroos, and S.~von Alfthan}, {\em Vlasov methods in space physics and astrophysics}, Living reviews in computational astrophysics, 4 (2018), p.~1.

\bibitem{comment}
{\sc H.~Qin, Y.~He, R.~Zhang, J.~Liu, J.~Xiao, and Y.~Wang}, {\em Comment on "{H}amiltonian splitting for the {V}lasov--{M}axwell equations"}, Journal of Computational Physics, 297 (2015), pp.~721--723.

\bibitem{jieshen}
{\sc J.~Shen, T.~Tang, and L.-L. Wang}, {\em Spectral methods: algorithms, analysis and applications}, vol.~41, Springer Science \& Business Media, 2011.

\bibitem{bookvlasov}
{\sc E.~Sonnendr{\"u}cker}, {\em Numerical methods for the {V}lasov-{M}axwell equations}, Springer, 2016.

\bibitem{semi1}
{\sc E.~Sonnendr{\"u}cker, J.~Roche, P.~Bertrand, and A.~Ghizzo}, {\em The semi-lagrangian method for the numerical resolution of the {V}lasov equation}, Journal of computational physics, 149 (1999), pp.~201--220.

\bibitem{chacon1}
{\sc A.~Stanier, L.~Chac{\'o}n, and G.~Chen}, {\em A fully implicit, conservative, non-linear, electromagnetic hybrid particle-ion/fluid-electron algorithm}, Journal of Computational Physics, 376 (2019), pp.~597--616.

\bibitem{cancellation}
{\sc A.~Stanier, L.~Chacon, and A.~Le}, {\em A cancellation problem in hybrid particle-in-cell schemes due to finite particle size}, Journal of Computational Physics, 420 (2020), p.~109705.

\bibitem{Strang}
{\sc G.~Strang}, {\em On the construction and comparison of difference schemes}, SIAM journal on numerical analysis, 5 (1968), pp.~506--517.

\bibitem{disp}
{\sc D.~Told, J.~Cookmeyer, P.~Astfalk, and F.~Jenko}, {\em A linear dispersion relation for the hybrid kinetic-ion/fluid-electron model of plasma physics}, New Journal of Physics, 18 (2016), p.~075001.

\bibitem{Tronci}
{\sc C.~Tronci}, {\em Hamiltonian approach to hybrid plasma models}, Journal of Physics A: Mathematical and Theoretical, 43 (2010), p.~375501.

\bibitem{udovicic2009calculation}
{\sc Z.~Udovicic}, {\em Calculation of the moments of the cardinal b-spline}, Sarajevo journal of mathematics, 5 (2009), pp.~291--297.

\bibitem{valentini}
{\sc F.~Valentini, P.~Tr{\'a}vn{\'\i}{\v{c}}ek, F.~Califano, P.~Hellinger, and A.~Mangeney}, {\em A hybrid-{V}lasov model based on the current advance method for the simulation of collisionless magnetized plasma}, Journal of Computational Physics, 225 (2007), pp.~753--770.

\bibitem{vlasiator}
{\sc S.~Von~Alfthan, D.~Pokhotelov, Y.~Kempf, S.~Hoilijoki, I.~Honkonen, A.~Sandroos, and M.~Palmroth}, {\em Vlasiator: First global hybrid-{V}lasov simulations of {E}arth's foreshock and magnetosheath}, Journal of Atmospheric and Solar-Terrestrial Physics, 120 (2014), pp.~24--35.

\bibitem{welling}
{\sc J.~S. Welling, W.~F. Eddy, and T.~K. Young}, {\em Rotation of 3d volumes by {F}ourier-interpolated shears}, Graphical Models, 68 (2006), pp.~356--370.

\bibitem{YinL}
{\sc D.~Winske, L.~Yin, N.~Omidi, H.~Karimabadi, and K.~Quest}, {\em Hybrid simulation codes: {P}ast, present and future-{A} tutorial}, Space plasma simulation,  (2003), pp.~136--165.

\bibitem{jianyuan}
{\sc J.~Xiao, H.~Qin, J.~Liu, Y.~He, R.~Zhang, and Y.~Sun}, {\em Explicit high-order non-canonical symplectic particle-in-cell algorithms for {V}lasov-{M}axwell systems}, Physics of Plasmas, 22 (2015).

\bibitem{xiong2014high}
{\sc T.~Xiong, J.-M. Qiu, Z.~Xu, and A.~Christlieb}, {\em High order maximum principle preserving semi-{L}agrangian finite difference {W}{E}{N}{O} schemes for the {V}lasov equation}, Journal of Computational Physics, 273 (2014), pp.~618--639.

\end{thebibliography}


\begin{thebibliography}{00}

\bibitem{yingzhe1} Li Y, Campos Pinto M, Holderied F, Possanner S,  Sonnendr\"ucker E. Geometric Particle-In-Cell discretizations of a plasma hybrid model with kinetic ions and mass-less fluid electrons. Journal of Computational Physics, 2024, 498, 112671.

\bibitem{Poisson} Li Y, Morrison P J, Possanner S, Sonnendr\"ucker E. On the Poisson brackets of the hybrid model with kinetic ions and massless electrons, 2024.

\bibitem{exact1} Bernier J, Crouseilles N,  Li Y. Exact splitting methods for kinetic and Schr\"odinger equations. Journal of Scientific Computing, 2021, 86, 1-35.

\bibitem{exact2} Bernier J. Exact splitting methods for semigroups generated by inhomogeneous quadratic differential operators. Foundations of Computational Mathematics, 2021, 21(5), 1401-1439.

\bibitem{exact3} Bernier J, Casas F,  Crouseilles N. Splitting methods for rotations: application to Vlasov equations. SIAM Journal on Scientific Computing, 2020, 42(2), A666-A697.

\bibitem{filbet} Filbet F, Sonnendr\"ucker E,  Bertrand P.  Conservative numerical schemes for the Vlasov equation. Journal of Computational Physics, 2001, 172(1), 166-187.

\bibitem{chenbao} Chen B, Kaufman A. 3D volume rotation using shear transformations. Graphical Models, 2000, 62(4), 308-322.

\bibitem{welling} Welling J S, Eddy W F, Young T K: Rotation of 3D volumes by Fourier-interpolated shears. Graph. Models 68(4), 2006, 356-370.



%\bibitem{3} Holderied F, Possanner S, Wang X. MHD-kinetic hybrid code based on structure-preserving finite elements with particles-in-cell. Journal of Computational Physics, 2021, 433: 110143.

\bibitem{yingzhe6d} Li Y, He Y, Sun Y, Niesen J, Qin H,  Liu J. Solving the Vlasov--Maxwell equations using Hamiltonian splitting. Journal of Computational Physics, 2019, 396, 381-399.

\bibitem{yingzhe2} Li Y, Holderied F, Possanner S, Sonnendr\"ucker E. Canonical variables based numerical schemes for hybrid plasma models with kinetic ions and massless electrons. Journal of Computational Physics, 2024, 505, 112916.

\bibitem{2020en} Kormann K, Sonnendr\"ucker E. Energy-conserving time propagation for a structure-preserving particle-in-cell Vlasov--Maxwell solver. Journal of Computational Physics, 2020, 425: 109890.

\bibitem{GEMPIC} Kraus M, Kormann K, Morrison P J,  Sonnendr\"ucker E. GEMPIC: geometric electromagnetic particle-in-cell methods. Journal of Plasma Physics, 2017, 83(4), 905830401.

\bibitem{jianyuan} Xiao J, Qin H, Liu J, He Y, Zhang R, Sun Y. Explicit high-order non-canonical symplectic particle-in-cell algorithms for Vlasov--Maxwell systems. Physics of Plasmas, 2015, 22(11).

\bibitem{heyang} He Y, Sun Y, Qin H,  Liu J. Hamiltonian particle-in-cell methods for Vlasov-Maxwell equations. Physics of Plasmas, 2016, 23(9).

\bibitem{struphy} Holderied F, Possanner S, Wang X.  MHD-kinetic hybrid code based on structure-preserving finite elements with particles-in-cell. Journal of Computational Physics, 2021, 433, 110143.

\bibitem{directional} Huot F, Ghizzo A, Bertrand P, Sonnendr\"ucker E, Coulaud O. Instability of the time splitting scheme for the one-dimensional and relativistic Vlasov-Maxwell system. Journal of Computational Physics, 2003, 185(2), 512-531.

\bibitem{martin} Campos Pinto M, Ameres J, Kormann K,  Sonnendr\"ucker E. On Variational Fourier Particle Methods. Journal of Scientific Computing, 2024, 101(3), 68.

%\bibitem{lukas} Einkemmer L,  Lubich C. A low-rank projector-splitting integrator for the Vlasov--Poisson equation. SIAM Journal on Scientific Computing, 2018, 40(5), B1330-B1360.

\bibitem{dual} Kormann K,  Sonnendr\"ucker E. A Dual Grid Geometric Electromagnetic Particle in Cell Method. SIAM Journal on Scientific Computing, 2024, 46(5), B621-B646.

\bibitem{disp} Told D, Cookmeyer J, Astfalk P, Jenko F. A linear dispersion relation for the hybrid kinetic-ion/fluid-electron model of plasma physics. New Journal of Physics, 2016, 18(7): 075001.

\bibitem{crouseilles2010conservative} Crouseilles N, Mehrenberger M, Sonnendr\"ucker, E. Conservative semi-Lagrangian schemes for Vlasov equations. Journal of Computational Physics, 2010, 229(6), 1927-1953.

%\bibitem{pic1} Birdsall C K, Langdon A B. Plasma physics via computer simulation. CRC press, 2018.

%\bibitem{pic2} Hockney R W, Eastwood J W. Computer simulation using particles[M]. CRC Press, 2021.

%\bibitem{Califano} Mangeney A, Califano F, Cavazzoni C,  Tr\'avn\'i\v{c}ek P. A numerical scheme for the integration of the Vlasov--Maxwell system of equations. Journal of Computational Physics, 2002, 179(2): 495-538.


\bibitem{hongtao} Liu H, Cai X, Cao Y,  Lapenta G. An efficient energy conserving semi-Lagrangian kinetic scheme for the Vlasov-Amp\`ere system. Journal of Computational Physics, 2023, 492, 112412.

\bibitem{chacon1} Stanier A, Chac\'on L, Chen G. A fully implicit, conservative, non-linear, electromagnetic hybrid particle-ion/fluid-electron algorithm. Journal of Computational Physics, 2019, 376: 597-616.

\bibitem{semi1} Sonnendr\"ucker, E, Roche, J, Bertrand, P, Ghizzo, A. The semi-Lagrangian method for the numerical resolution of the Vlasov equation. Journal of computational physics, 1999, 149(2): 201-220.

\bibitem{semi2} Cheng C Z, Knorr G. The integration of the Vlasov equation in configuration space. Journal of Computational Physics, 1976, 22(3), 330-351.

\bibitem{cancellation} Stanier A, Chac\'on L,  Le A. A cancellation problem in hybrid particle-in-cell schemes due to finite particle size. Journal of Computational Physics, 2020, 420, 109705.

\bibitem{vlasiator} Von Alfthan, S, Pokhotelov D, Kempf Y, Hoilijoki S, Honkonen I, Sandroos A, Palmroth M. Vlasiator: First global hybrid-Vlasov simulations of Earth's foreshock and magnetosheath. Journal of Atmospheric and Solar-Terrestrial Physics, 2014, 120, 24-35.


\bibitem{current2D} Matthews Alan P. Current advance method and cyclic leapfrog for 2D multispecies hybrid plasma simulations. Journal of Computational Physics, 1994, 112(1): 102-116.

\bibitem{palmroth2018vlasov} Palmroth M, et al. Vlasov methods in space physics and astrophysics. Living reviews in computational astrophysics, 2018, 4(1), 1.

\bibitem{valentini} Valentini F, Tr\'{a}vn\'{i}\v{c}ek P, Califano F, Hellinger P, Mangeney A. A hybrid-Vlasov model based on the current advance method for the simulation of collisionless magnetized plasma. Journal of Computational Physics, 2007, 225(1): 753-770.

\bibitem{CAMELIA} Franci L, Hellinger P, Guarrasi M, et al. Three-dimensional simulations of solar wind turbulence with the hybrid code CAMELIA. Journal of Physics: Conference Series. IOP Publishing, 2018, 1031(1): 012002.

\bibitem{Pegasus} Kunz M W, Stone J M, Bai X N. Pegasus: a new hybrid-kinetic particle-in-cell code for astrophysical plasma dynamics. Journal of Computational Physics, 2014, 259: 154-174.

%\bibitem{Servidio} Servidio S, Valentini F, Califano F, et al. Local kinetic effects in two-dimensional plasma turbulence[J]. Physical review letters, 2012, 108(4): 045001.

%\bibitem{mercury} Fatemi S, Poppe A R, Barabash S. Hybrid simulations of solar wind proton precipitation to the surface of Mercury. Journal of Geophysical Research: Space Physics, 2020, 125(4): e2019JA027706.

%\bibitem{Schekochihin} Kunz M W, Schekochihin A A, Stone J M. Firehose and mirror instabilities in a collisionless shearing plasma[J]. Physical Review Letters, 2014, 112(20): 205003.


%\bibitem{densityjapan} Amano T, Higashimori K, Shirakawa K. A robust method for handling low density regions in hybrid simulations for collisionless plasmas. Journal of Computational Physics, 2014, 275: 197-212.

%\bibitem{CHIEF} Mu\~{n}oz P A, Jain N, Kilian P,  B\"uchner J. A new hybrid code (CHIEF) implementing the inertial electron fluid equation without approximation. Computer Physics Communications, 2018, 224: 245-264.

\bibitem{Tronci}{Tronci C. Hamiltonian approach to hybrid plasma models. Journal of Physics A, 2010, 43(37).}

\bibitem{Strang} Strang G. On the construction and comparison of difference schemes. SIAM journal on numerical analysis, 1968, 5(3), 506-517.

\bibitem{iter} Kelley C T. Iterative methods for linear and nonlinear equations. Society for Industrial and Applied Mathematics, 1995.

\bibitem{splitting} McLachlan R I,  Quispel G R W.  Splitting methods. Acta Numerica, 2002, 11, 341-434.

\bibitem{bookvlasov} Sonnendr{\"u}cker Eric.  Numerical methods for the Vlasov-Maxwell equations. 2016, Springer.

\bibitem{he2015volume} He Y, Sun Y, Liu J, Qin H. Volume-preserving algorithms for charged particle dynamics. Journal of Computational Physics, 2015, 281, 135-147.

%\bibitem{Feng} Feng K, Qin M. Symplectic geometric algorithms for Hamiltonian systems. Berlin: Springer, 2010.

\bibitem{HLW} Hairer E, Lubich C, Wanner G. Geometric Numerical Integration: Structure-Preserving Algorithms for Ordinary Differential Equations, vol. 31, Springer Science \& Business Media, 2006.

\bibitem{feec} Arnold D N, Falk R S, Winther R. Finite element exterior calculus, homological techniques, and applications. 2006, Acta numerica, 15, 1-155.

\bibitem{dec} Desbrun M, Hirani A N, Leok M, Marsden J E. Discrete exterior calculus. 2005, arXiv preprint math/0508341.

\bibitem{chaconpara} Chac\'on L, Knoll D A,  Finn J M. An implicit, nonlinear reduced resistive MHD solver. Journal of Computational Physics, 2002, 178(1), 15-36.


\bibitem{le} Le A, Stanier A, Yin L, Wetherton B, Keenan B,  Albright B. Hybrid-VPIC: An open-source kinetic/fluid hybrid particle-in-cell code. Physics of Plasmas, 20203, 30(6).

%\bibitem{chacon2} Stanier A, Chac\'on L. A conservative implicit-PIC scheme for the hybrid kinetic-ion fluid-electron plasma model on curvilinear meshes. Journal of Computational Physics, 2022, 459: 111144.

\bibitem{Nicolas} Crouseilles N, Einkemmer L, Faou E. Hamiltonian splitting for the Vlasov--Maxwell equations. Journal of Computational Physics, 2015, 283: 224-240.

\bibitem{he} He Y, Qin H, Sun Y, Xiao J, Zhang R, Liu J. Hamiltonian time integrators for Vlasov-Maxwell equations. Physics of Plasmas, 2015, 22(12).

\bibitem{comment} Qin H, He Y, Zhang R, Liu J, Xiao J, Wang Y. Comment on 'Hamiltonian splitting for the Vlasov--Maxwell equations'. Journal of Computational Physics, 2015, 297, 721-723.

\bibitem{udovicic2009calculation} Udovicic Z. Calculation of the moments of the cardinal B-spline. Sarajevo journal of mathematics, 2009, 5(18), 291-297.
 
%\bibitem{Trotter} Trotter H F. On the product of semi-groups of operators. Proceedings of the American Mathematical Society, 1959, 10(4): 545-551.


%\bibitem{Rambo} Rambo P W. Finite-grid instability in quasineutral hybrid simulations. Journal of Computational Physics, 1995, 118(1): 152-158.

%\bibitem{potential} Li Y, Holderied F, Possanner S, Sonnendr\"ucker E. Canonical momentum based numerical schemes for hybrid plasma models with kinetic ions and massless electrons, submitted. 

\bibitem{cai} Cai H B, Yan X X, Yao P L, Zhu S P. Hybrid fluid-particle modeling of shock-driven hydrodynamic instabilities in a plasma. Matter and Radiation at Extremes, 2021 6(3).

\bibitem{YinL} Winske D, Yin L, Omidi N, et al. Hybrid simulation codes: Past, present and future-A tutorial. Space plasma simulation, 2003: 136-165.

\bibitem{jieshen} Shen J, Tang T,  Wang L L. Spectral methods: algorithms, analysis and applications (Vol. 41). Springer Science \& Business Media, 2011.
 
\end{thebibliography}

\begin{comment}

\end{comment}

\end{document}